\documentclass[10pt,leqno]{amsart}
\usepackage{graphicx}
\usepackage{indentfirst,csquotes}

\usepackage{amssymb,amsthm,amsmath}
\usepackage{xcolor,paralist,hyperref,titlesec,fancyhdr,etoolbox}
\newtheorem{theorem}{Theorem}[]

\newtheorem{lemma}[theorem]{Lemma}

\newtheorem{corollary}[theorem]{Corollary}

\titleformat{\section}[display]{\normalfont\huge\bfseries\centering}{\centering\chaptertitlename\thechapter}{10pt}{\Large}
\titlespacing*{\section}{0pt}{0ex}{0ex}

\hypersetup{ colorlinks=true, linkcolor=black, filecolor=black, urlcolor=black }

\usepackage{lipsum}

\begin{document}
\title{Mitchell Rank for Supercompactness} %%%%%%%%%%%%
\author[Initial Surname]{Erin Carmody}
\date{\today}
\address{}
\email{carmodek@delhi.edu}
\maketitle

\let\thefootnote\relax
\footnotetext{} %%%%%%%%%%

\begin{abstract}
This paper defines a Mitchell rank for supercompact cardinals. If $\kappa$ is a $\theta$-supercompact cardinal then $o_{\theta-sc}(\kappa) = \sup \{ o_{\theta-sc}(\mu) + 1 \ | \ \mu \in m(\kappa)\}$, where $m(\kappa)$ is the collection of normal fine measures on $P_{\kappa}\theta$. We show how to force to kill the degree of a measurable cardinal $\kappa$ to any specified degree which is less than or equal to the degree of $\kappa$ in the ground model. We will also show how to softly kill the Mitchell rank for supercompactness of any supercompact cardinal so that in the forcing extension it is any desired degree less than or equal to its degree in the ground model, along with some results concerning strongly compact cardinals. 
\end{abstract} %%%%%%%%%

\bigskip
\begin{center}\textbf{INTRODUCTION}\end{center}
\vspace{.2 in}

Almost every large cardinal studied so far comes with degrees of its existence, or seems to have the potential for degrees. For example, Mitchell rank for measurable cardinals provides degrees for measurable cardinals. In a previous paper [2], I showed how to precisely define the degrees of inaccessible cardinals and Mahlo cardinals. In another paper, Gitman and Habič [3] define degrees of Ramsey cardinals. This paper defines a Mitchell rank for supercompact cardinals. Along with the degrees of a large cardinal, comes the forcings which can kill a large cardinal to any desired degree. It is not yet known if for every degree of a large cardinal there is a forcing to kill it so that in the forcing extension it has the desired maximum rank or degree, but so far they seem to go hand in hand. In this paper, I will also show how to force to cut down a measurable cardinal's Mitchell rank to any (possible) desired rank, and then we will see how to softly kill Mitchell rank for supercompactness.

Before we begin with the definition of measurable cardinals and Mitchell rank for measurable and supercompact cardinals, let us see some theorems which fit into the killing-them-softly theme including a result by Hamkins and Shelah which shows how to force to kill the already known degrees of supercompact cardinals (the Mitchell rank for supercompactness further differentiates levels of supercompactness).

\begin{theorem}[Schanker]  If $\kappa$ is weakly measurable, then the measurability of $\kappa$ can be destroyed while preserving that $\kappa$ is weakly measurable. 
\end{theorem}

\begin{theorem}[Gitman]  If $\kappa$ is strongly Ramsey, then there is a forcing extension where $\kappa$ is not weakly measurable but is still strongly Ramsey.
\end{theorem}

\begin{theorem}[Gitman]  If $\kappa$ is strongly Ramsey and ineffable, then there is a forcing extension where $\kappa$ is not ineffable, but is still strongly Ramsey.
\end{theorem}

\begin{theorem}[Hamkins/Shelah]  If $\kappa$ is $\theta$-supercompact where $\kappa < \theta$ and $\theta^{< \kappa} = \theta$, then there is a forcing extension where $\kappa$ is $\theta$-supercompact but not $\theta^{+}$-supercompact.
\end{theorem}

\begin{theorem}[Magidor]  If $\kappa$ is strongly compact, then there is a forcing extension where $\kappa$ is strongly compact, but not supercompact. 
\end{theorem}

\vspace{.5 in}

\begin{center}\textbf{MEASURABLE CARDINALS} \end{center}
\vspace{.2 in}

A cardinal $\kappa$ is \textit{measurable} if and only if $\kappa$ is uncountable and there is a $\kappa$-complete nonprincipal ultrafilter on $\kappa$. Suppose $\kappa$ is measurable by measure $\mu$. Then, the corresponding ultrapower embedding $j:V \to M_{\mu}$, by measure $\mu$ is an elementary embedding with critical point $\kappa$. Also, if $\kappa$ is the critical point of a non-trivial elementary embedding, then it is measurable. Thus, another characterization of measurable cardinals is the following embedding definition: $\kappa$ is measurable if and only if it is the critical point of an elementary embedding $j:V \to M$ in $V$. 

When we begin to ask whether $\kappa$ is measurable in $M$, we enter the realm of Mitchell order. The theorems in this section are about reducing Mitchell rank for measurable cardinals in a forcing extension. The Mitchell order is defined on normal measures; for normal measures $\mu$ and $\nu$, the relation $\mu \vartriangleleft \nu$ holds when $\mu \in M_{\nu}$, where $M_{\nu}$ is the ultrapower by $\nu$. If $\kappa$ is measurable, then the Mitchell order on the measures on $\kappa$ is well-founded (and transitive) [Jech p. 358]. Thus, let $m(\kappa)$ be the collection of normal measures on $\kappa$. The definition of rank for $\mu \in m(\kappa)$ with respect to $\vartriangleleft$ is $o(\mu) = \sup \{o(\nu) + 1 \ | \ \nu \vartriangleleft \mu\}$. The Mitchell order on $\kappa$, denoted $o(\kappa)$, is the height of the Mitchell order on $m(\kappa)$, thus $o(\kappa) = \sup\{o(\mu) + 1 \ | \ \mu \in m(\kappa)\}$. If $\kappa$ is not measurable, then $o(\kappa) = 0$, since $m(\kappa)$ is empty. A cardinal $\kappa$ is measurable if and only if $o(\kappa) \ge 1.$ For $o(\kappa) \ge 2$, there is a measure $\mu$ on $\kappa$, such that if $j:V \to M_{\mu}$ is the corresponding ultrapower embedding, there is a measure on $\kappa$ in $M_{\mu}$. Thus $\kappa$ is measurable in $M_{\mu}$.  Thus $M_{\mu} \models \kappa \in j(\{ \gamma < \kappa \ | \ \gamma \mbox{ is measurable} \} )$, and so the set $\{ \gamma < \kappa \ | \ \gamma \mbox{ is measurable } \}$ is in $\mu$. That is, $\mu$ concentrates on measurable cardinals. Therefore, $o(\kappa) \ge 2$ if and only if there is a normal measure on $\kappa$ which concentrates on measurable cardinals. A cardinal $\kappa$ has Mitchell order $o(\kappa) \ge 3$, if and only if there is a normal measure on $\kappa$ which concentrates on measurable cardinals with Mitchell order 2. A cardinal $\kappa$ has $o(\kappa) \ge \alpha$ if and only if for every $\beta < \alpha,$ there is a normal measure on $\kappa$ which concentrates on measurable cardinals of order $\beta$.

The following theorem shows to force any measurable cardinal, of any Mitchell rank, to have order exactly one. In the proof of the following theorem, and the proof of the main theorems, we will need that the forcing notions do not create large cardinals. In [2] since being Mahlo or inaccessible is downwards absolute, we knew that the forcing could not create new Mahlo or inaccessible cardinals. However, it is possible that forcing creates new large cardinals in the extension. Thankfully, we have the approximation and cover properties between models and so by Hamkins [11] we can ensure that no new measurable cardinals are created. 
\begin{theorem} If $\kappa$ is a measurable cardinal, then there is a forcing extension where $o(\kappa) = 1$. In addition, the forcing will preserve all measurable cardinals and will not create any new measurable cardinals.\\
\end{theorem}

\begin{proof} Suppose $\kappa$ is a measurable cardinal. Assume that $2^{\kappa} = \kappa^+$, or force it to be true since forcing to have the GCH hold at $\kappa$ preserves the measurability of $\kappa$. Pick $j:V \to M$, an elementary ultrapower embedding by $\mu$, a normal measure on $\kappa$ which has minimal Mitchell rank $o(\mu) = 0$, with $cp(j) = \kappa$, such that $M \models \kappa$ is not measurable. Let $\mathbb{P}$ be a $\kappa$-length iteration, with Easton support, of $\mathbb{Q}_{\gamma}$, a $\mathbb{P}_{\gamma}$-name for the forcing to add a club to $\gamma$ of cardinals which are not measurable in $V$, whenever $\gamma < \kappa$ is inaccessible, otherwise force trivially at stage $\gamma$. Let $G \subseteq \mathbb{P}$ be $V$-generic, let $\dot{\mathbb{Q}}_G = \mathbb{Q}$, and let $\dot{\mathbb{Q}}$ be a $\mathbb{P}$-name for a forcing to add a club of cardinals to $\kappa$ which are not measurable in $V[G]$. Let $g \subseteq \mathbb{Q}$ be $V[G]$-generic. Let $\delta_0$ be the first inaccessible cardinal. This is the first non-trivial stage of $\mathbb{P}$, and at this stage a club is added to $\delta_0$. Thus the forcing up to and including stage $\delta_0$ has cardinality at most $\delta_0$ and the forcing after stage $\delta_0$ is closed up to the next inaccessible cardinal. Thus $\mathbb{P}$ has a closure point at $\delta_0$. Thus by Corollary 22 in Hamkins [11], the forcing does not create measurable cardinals.

We want to lift $j$ through $\mathbb{P}$ and $\mathbb{Q}$ to find $j:V[G][g] \to M[j(G)][j(g)]$ which witnesses that $\kappa$ is measurable in $V[G][g]$. The main task is to find $j(G)$, and $M$-generic filter for $j(\mathbb{P})$, and $j(g)$ an $M[j(G)]$-generic filter for $j(\mathbb{Q})$. Since the critical point of $j$ is $\kappa$, the forcings $\mathbb{P}$ and $j(\mathbb{P})$ are isomorphic up to stage $\kappa$. Also, since $\dot{\mathbb{Q}}$ is a $\mathbb{P}$-name for a forcing which adds a club of ground model non-measurable cardinals to $\kappa$ and $V[G]$ and $M[G]$ have the same bounded subsets of $\kappa$, the forcing at stage $\kappa$ of $j(\mathbb{P})$ is $\dot{\mathbb{Q}}$.  Thus, $j(\mathbb{P})$ factors as $\mathbb{P}\cdot\dot{\mathbb{Q}}\cdot\mathbb{P}_{\mbox{tail}}$, where $\mathbb{P}_{\mbox{tail}}$ is the forcing $j(\mathbb{P})$ past stage $\kappa$. Since $G$ is $V$-generic, it is also $M$-generic. Thus, form the structure $M[G]$. Similarly, since $g$ is $V[G]$-generic, it follows that $g$ is $M[G]$-generic, so construct the structure $M[G][g]$. Since both $\mathbb{P}$ and $\mathbb{Q}$ have size $\kappa$, they both have the $\kappa^+$-chain condition, so the iteration $\mathbb{P}*\mathbb{Q}$ has the $\kappa^+$-chain condition. Also, since $j''\kappa = \kappa \in M$, and $j$ is an ultrapower embedding, it follows that $M^{\kappa} \subseteq M$ in $V$. Thus, since $G*g$ is $V$-generic, $M[G][g]^{\kappa} \subseteq M[G][g]$ by Hamkins [8 (Theorem 54)]. 

The goal is to diagonalize to find and $M[G][g]$-generic filter for $\mathbb{P}_{\mbox{tail}}$ in $V[G][g]$, and so far we have satisfied one criterion which is that $M[G][g]$ is closed under $\kappa$ sequences in $V[G][g]$. Next, since $|\mathbb{P}| = \kappa$ and $2^{\kappa} = \kappa^+$, it follows that $\mathbb{P}$ has at most $\kappa^+$ many dense sets, and so does any tail $\mathbb{P}_{[\beta, \kappa)}$. Thus, $\mathbb{P}_{\mbox{tail}}$ has at most $|j(\kappa^+)|^{V} \le \kappa^{+^{\kappa}} = \kappa^+$ dense subsets in $M[G][g]$. Finally, since for every $\beta < \kappa$, there is, in $M[G][g]$, a dense subset of $\mathbb{P}$ which is $\le \beta$-closed, there is a dense subset of $\mathbb{P}_{\mbox{tail}}$ which is $\le \kappa$-closed. Thus, it is possible to diagonalize to get an $M[G][g]$-generic filter $G_{\mbox{tail}} \subseteq \mathbb{P}_{\mbox{tail}}$, in $V[G][g]$. Thus, $j(G) = G*g*G_{\mbox{tail}} \subseteq j(\mathbb{P})$ is $M$-generic, and $j''G \subseteq G*g*G_{\mbox{tail}}$. Therefore, the lifting criterion is satisfied, and so the embedding lifts to $j:V[G] \to M[j(G)]$.

The next goal is to lift the embedding through $\mathbb{Q}$. The forcing $j(\mathbb{Q})$ is the forcing to a club of non-measurables to $j(\kappa)$, and we need an $M[j(G)]$-generic filter for $j(\mathbb{Q})$. Since $G_{\mbox{tail}} \in V[G][g]$, it follows that $M[j(G)]^{\kappa} \subseteq M[j(G)]$. Also, since $|\mathbb{Q}| = \kappa$, it has at most $\kappa^+$ many dense sets in $V[G]$; by elementarity $j(\mathbb{Q})$ has at most $|j(\kappa^+)|^{V[G]} \le \kappa^{+^{\kappa}}$ many dense subsets in $M[j(G)]$. For any $\beta < \kappa$, the forcing $\mathbb{Q}$ has a dense subset which is $\le \beta$-closed. Since $\kappa < j(\kappa)$, the forcing $j(\mathbb{Q})$ has a dense subset which is $\le \kappa$-closed. Note that $c = \bigcup g$, which is club in $\kappa$, is in $M[j(G)$. Let $\overline{c} = c \ \cup \ \{\kappa\}$. For all $\delta < \kappa$, we have if $\delta \in c$ then $\delta$ is not measurable in $V$. Thus, $\delta \in c$ implies $\delta$ is not measurable in $M$, since $M$ and $V$ have the same $P(P(\delta))$ since $\delta < \kappa$. By Hamkins [11 (Corollary 22)] since $M \subseteq M[j(G)]$ satisfies the approximation and cover properties, it follows that $\delta$ is not measurable in $M[j(G)]$. By choice of $j:V \to M$, we have that $\kappa$ is not measurable in $M$. Thus $\kappa$ is not measurable in $M(j(G))$ again by Hamkins [11]. Thus $\overline{c}$ is a closed, bounded subset of $j(\kappa)$ which contains no measurable cardinals. Thus $\overline{c} \in j(\mathbb{Q})$. Thus, diagonalize to get a generic filter $g^* \subseteq j(\mathbb{Q})$, in $V[G]$, which meets a dense subset which is $\le \kappa$-closed, and which contains $\overline{c}$. Thus, let $j(g) = g^*$, and observe that $j''g \subseteq g^*$ since $g \in \bigcup g^*$. Thus, the embedding lifts to $j:V[G][g] \to M[j(G)][j(g)]$, in $V[G][g]$. Therefore, $\kappa$ is still measurable in $V[G][g]$. Thus $o(\kappa) \ge 1$ in $V[G][g]$.

Finally, $o(\kappa)^{V[G][g]} \le 1$ will follow from the fact that $\mathbb{Q}$ adds a club subset $c \subseteq \kappa$ which contains no cardinals which are measurable in $V$. Since the combined forcing does not create measurable cardinals, the club $c$ contains no cardinals which are measurable in $V[G][g]$. Since the new club has measure one in any normal measure on $\kappa$, the complement of $c$, which contains the measurable cardinals of $V[G][g]$ below $\kappa$ has measure zero. Thus, there is no normal measure on $\kappa$ in $V[G][g]$ which concentrates on the measurable cardinals of $V[G][g]$. Hence $o(\kappa)^{V[G][g]} \ngtr 1$. Thus $o(\kappa) = 1$ in $V[G][g]$. That the forcing preserves all measurable cardinals and creates no new measurable cardinals will follow from the proof of the more general theorem.
\end{proof}

The previous theorem shows how to make the Mitchell order of any measurable cardinal exactly one, for any measurable cardinal. The next theorem will show how to make the Mitchell order of any measurable cardinal exactly $\alpha$, for any $\alpha < \kappa$, in a forcing extension, whenever the order is at least $\alpha$ in the ground model. Preceeding the proof are the following lemmas, so that we may generalize the proof for any $\alpha < \kappa$.

\begin{lemma} Let $\kappa$ be an infinite cardinal with $\kappa^{< \kappa} = \kappa$. Let $S \subseteq \kappa$ be a subset of $\kappa$ which contains the singular cardinals. Then $\mathbb{Q}_S$, the forcing to add a club $C \subset S$, preserves cardinals and cofinalities, and for all $\beta < \kappa$, the forcing $\mathbb{Q}_S$ has a $\le \beta$-closed dense subset. 
\end{lemma}

\begin{proof} Conditions $c \in \mathbb{Q}_S$ are closed, bounded subsets of $S$. The ordering on $\mathbb{Q}_S$ is end extension: $c \le d$ if and only if $a \in c \setminus d \implies a > \max(d)$. Let $G \subseteq \mathbb{Q}_S$ be $V$-generic, and let $C = \bigcup G$. Then $C \subseteq S$ is club in $\kappa$. First, see that $C$ is unbounded. Let $\beta \in \kappa$, and let $D_{\beta} = \{ d \in \mathbb{Q}_S \ : \ \max(d) > \beta\}$. Then, $D_{\beta}$ is dense in $\mathbb{Q}_S$, since if $c \in \mathbb{Q}_S$, the following is a condition: $d = c \cup \{\beta'\}$, where $\beta' > \beta$ and $\beta'$ is in $S$. Then, $d$ is an end-extension of $c$, hence stronger, and $d \in D_{\beta}$, shows $D_{\beta}$ is dense in $\mathbb{Q}_S$. Hence, there is $c_{\beta} \in D_{\beta} \cap G$, so that there is an element of $C$ above $\beta$. Thus, $C$ is unbounded in $\kappa$. Next, see that $C$ is closed. Suppose $\delta \cap C$ is unbounded in $\delta < \kappa$. Since $C$ is unbounded, there is $\delta' \in C$, where $\delta' > \delta$, and a condition $c \in C$ which contains $\delta'$. Since the conditions are ordered by end-extension, and since there is a sequence (possibly of length 1) of conditions which witness that $\delta \cap C$ is unbounded in $C$, which will all be contained in the condition $c$, which contains an element above $\delta$, it follows that $\delta \cap c$ is unbounded in $\delta$. Since $c$ is closed, $\delta \in c$. Therefore $\delta \in C$, which shows $C$ is closed.

Furthermore, $\mathbb{Q}_S$ preserves cardinals and cofinalities $\ge \kappa^+$ since $|\mathbb{Q}_S| = \kappa^{< \kappa} = \kappa$. If $\beta < \kappa$, then $\mathbb{Q}_S$ is not necessarily $\le \beta$-closed. Let $\beta$ be the first limit of $S$ which is not in $S$. Then, the initial segments of $S$ below $\beta$ are in $\mathbb{Q}_S$, and union up to $S \cap \beta$, so that no condition can close this $\beta$-sequence since $\beta \in S$. However, for every $\beta < \kappa$ the set $D_{\beta} = \{ d \in \mathbb{Q}_S \ : \ \max(d) \ge \beta \}$ is dense in $\mathbb{Q}_S$ and is $\le \beta$-closed since $\mbox{cof($\kappa$)} > \beta$. Thus, for every $\beta < \kappa$, the poset $\mathbb{Q}_S$ is forcing equivalent to $\mathbb{C}_{\beta}$, which is $\le \beta$-closed. Thus, $\mathbb{Q}_S$ preserves all cardinals, cofinalities, and strong limits $\le \kappa$.
\end{proof}

The following lemma is about finding a witness for the Mitchell rank of each normal measure $\mu$ on a measurable cardinal $\kappa$. By definition $o(\kappa)$ is the height of the well-founded Mitchell relation $\vartriangleleft$ on $m(\kappa)$, the collection of normal measures on $\kappa$. Thus we have seen that $o(\kappa) = \{ o(\mu) + 1 \ | \ \mu \in m(\kappa)\}$. Similarly we recall the definition of rank for $\mu \in m(\kappa)$ with respect to $\vartriangleleft$ is $o(\mu) = \sup\{o(\nu) + 1 \ | \ \nu \vartriangleleft \mu \}$.

\begin{lemma} If $\kappa$ is a measurable cardinal, and $\mu$ is a normal measure on $\kappa$, and $j_{\mu} : V \to M_{\mu}$ is the ultrapower embedding by $\mu$ with critical point $\kappa$, then $o(\mu) = o(\kappa)^{M_{\mu}}$. Thus, $\beta < o(\kappa)$ if and only if there exists $j : V \to M$ elementary embedding with critical point $\kappa$ with $o(\kappa)^M = \beta$.
\end{lemma}

\begin{proof} Suppose $\kappa$ is measurable, $\mu$ is a normal ultrafilter on $\kappa$, and $j_{\mu} : V \to M_{\mu}$ is the ultrapower embedding by $\mu$ with critical point $\kappa$. The model $M_{\mu}$ computes $\vartriangleleft$ correctly for $\nu \vartriangleleft \mu$, which means $\nu \in M_{\mu}$, since $M^{\kappa} \subseteq M$ and so $M_{\mu}$ has all the functions $f: \kappa \to V_{\kappa}$ needed to check whether $\nu \vartriangleleft \mu$. Then
\begin{equation}
\begin{split}
o(\kappa)^{M_{\mu}} & = \sup\{o(\nu)^{M_{\mu}} + 1 \ | \ \nu \in M_{\mu} \} \\
 & = \sup\{o(\nu)^{M_{\mu}}+1 \ | \ \nu \vartriangleleft \mu\}\\
 & = \sup\{o(\nu) + 1 \ | \ \nu \vartriangleleft \mu\} \\
 & = o(\mu).
\end{split}
\end{equation}

Thus, the Mitchell rank of $\kappa$ in $M_{\mu}$ is the Mitchell rank of $\mu$. Let $\beta < o(\kappa)$. Since $\vartriangleleft$ is well-founded, there is $\mu$ a normal measure on $\kappa$ such that $o(\mu) = \beta$. Then if $j_{\mu} : V \to M_{\mu}$ is the ultrapower by $\mu$ with critical point $\kappa$, then $o(\kappa)^{M_{\mu}} = o(\mu) = \beta$, as desired. Conversely, suppose there is an elementary embedding $j: V \to M$ with cp$(j) = \kappa$ and $o(\kappa)^M = \beta$. Let $\mu$ be the induced normal measure, using $\kappa$ as a seed ($X \in \mu$ if and only if $\kappa \in j(X)$) and let $M_{\mu}$ be the ultrapower by $\mu$. Then $M$ and $M_{\mu}$ have the same normal measures and the same Mitchell order and ranks, so $M_{\mu} \models o(\kappa) = \beta$. Since we have established $o(\kappa)^{M_{\mu}} = o(\mu)$, the Mitchell rank of $\mu$ in $m(\kappa)$ is $\beta$, that is $o(\mu) = \beta$. Thus since $o(\kappa) = \{ o(\mu) + 1 \ | \ \mu \in m(\kappa)\}$, it follows that $o(\kappa)^V > \beta$.
\end{proof}

The following lemma generalizes Corollary 2 of Hamkins [11] by showing that if $V \subseteq \overline{V}$ satisfies the $\delta$ approximation and cover properties, then the Mitchell rank of a cardinal $\kappa > \delta$ cannot increase between $V$ and $\overline{V}$.

\begin{lemma} Suppose $V \subseteq \overline{V}$ satisfies the $\delta$ approximation and cover properties. If $\kappa > \delta$ then $o(\kappa)^{\overline{V}} \le o(\kappa)^V$.
\end{lemma}

\begin{proof} Suppose $V \subseteq \overline{V}$ satisfies the $\delta$ approximation and cover properties. Suppose $\kappa > \delta$. We want to show $o(\kappa)^{\overline{V}} \le o(\kappa)^V$ by induction on $\kappa$. Assume inductively that for all $\kappa' < \kappa$, we have $o(\kappa')^{\overline{V}} \le o(\kappa')^V$. If $o(\kappa)^{\overline{V}} > o(\kappa)$ then by Lemma 8 get elementary ultrapower embedding $j:\overline{V} \to \overline{M}$ with critical point $\kappa$ and $o(\kappa)^{\overline{M}} = o(\kappa)^V$. By Hamkins [11], $j$ is a lift of an ultrapower embedding $j \upharpoonright V: V \to M$ for some $M$, and this embedding is a class in $V$. Again by Lemma 8 we have established that $o(\kappa)^M < o(\kappa)^V$. Thus, $o(\kappa)^M < o(\kappa)^{\overline{M}}$. However, this contradicts that $j$ applied to the induction hypothesis gives $o(\kappa)^{\overline{M}} \le o(\kappa)^M$ since $j < j(\kappa)$. Therefore, $o(\kappa)^{\overline{V}} \le o(\kappa)^V.$
\end{proof}

\begin{lemma} If $V \subseteq V[G]$ is a forcing extension and every normal ultrapower embedding $j:V \to M$ in $V$ (with any critical point $\kappa$) lifts to an embedding $j:V[G] \to M[j(G)]$, then $o(\kappa)^{V[G]} \ge o(\kappa)^V$. In addition, if $V \subseteq V[G]$ also has $\delta$ approximation and $\delta$ cover properties, then $o(\kappa)^{V[G]} = o(\kappa)^V$.
\end{lemma}

\begin{proof} Suppose $V \subseteq V[G]$ is a forcing extension where every normal ultrapower embedding in $V$ lifts to a normal ultrapower embedding in $V[G]$. Suppose $\kappa$ is a measurable cardinal, and for every measurable cardinal $\kappa' < \kappa$, assume $o(\kappa')^{V[G]} \ge o(\kappa')^V$. Fix any $\beta < o(\kappa)$. By Lemma 8 there exists an elementary embedding $j:V \to M$ with cp$(j)= \kappa$ and $M \models o(\kappa) = \beta$. Then, by the assumption on $V \subseteq V[G]$, this lifts $j$ to $j:V[G] \to M[j(G)]$. Then, applying $j$ to the induction hypothesis gives $o(\kappa)^{M[j(G)]} \ge o(\kappa)^M$ since $\kappa < j(\kappa)$. Then, since $M \models o(\kappa) = \beta$ it follows that $o(\kappa)^{M[j(G)]} \ge \beta$. Thus by Lemma 8 $o(\kappa)^{V[G]} > \beta$. Thus $o(\kappa)^{V[G]} \ge o(\kappa)^V$.

If $V \subseteq V[G]$ also satisfies the $\delta$ approximation and $\delta$ cover properties, then by Lemma 9 the Mitchell rank does not go up between those models. Thus $o(\kappa)^{V[G]} = o(\kappa)^V$.
\end{proof}

\begin{theorem} For any $V \models$ ZFC + GCH and any ordinal $\alpha$, there is a forcing extension $V[G]$ where every cardinal $\kappa$ above $\alpha$ in $V[G]$ has $o(\kappa)^{V[G]} = \min\{\alpha, o(\kappa)^V\}.$ In particular, if $1 \le \alpha$, then the forcing preserves all measurable cardinals of $V$ and creates no new measurable cardinals. 
\end{theorem}

\begin{proof} Let $\alpha \in$ ORD. The following notion of forcing will ensure that the Mitchell rank of all cardinals above $\alpha$ which have Mitchell rank above $\alpha$ will be softly killed to rank $\alpha$. Let $\mathbb{P}$ be an Easton support ORD-length iteration, forcing at inaccessible stages $\gamma$, to add a club $c_{\gamma} \subseteq \gamma$ such that $\delta \in c_{\gamma}$ implies $o(\delta)^V < \alpha$. Let $G \subseteq \mathbb{P}$ be $V$-generic. Let $\delta_0$ denote the first inaccessible cardinal. The forcing up to this stage, $\mathbb{P}_{\delta_0}$, adds nothing and the forcing $\mathbb{Q}$ at stage $\delta_0$ adds a club to $\delta_0$, so that $|\mathbb{P}_{\delta_0} * \dot{\mathbb{Q}}| \le \delta_0$. The forcing past stage $\delta_0$ is closed up to the next inaccessible cardinal. Since $\mathbb{P}$ has a closure point at $\delta_0$, by Lemma 13 in Hamkins [11] it follows that $V \subseteq V[G]$ satisfies the $\delta_0^+$ approximation and cover properties. Thus, by Lemma 9 the Mitchell rank does not go up between $V$ and $V[G]$. Note that any measurable cardinal is above $\delta_0^+$. 

Let $\kappa$ be a measurable cardinal above $\alpha$. Fix any $\beta_0 < \min\{\alpha, o(\kappa)^V\}$. The induction hypothesis is that for all $\kappa' < \kappa$, we have $o(\kappa')^{V[G]} = \min\{ \alpha, o(\kappa')^V\}.$ By Lemma 8, fix an elementary embedding $j:V \to M$ with cp$(j) = \kappa$ and $o(\kappa)^M = \beta_0$. We shall lift $j$ through $\mathbb{P}$. The forcing above $\kappa$ does not affect $o(\kappa)$ since the forcing past stage $\kappa$ has a dense subset which is closed up to the next inaccessible. Call $\mathbb{P}_{\kappa}*\dot{\mathbb{Q}}$ the forcing up to and including stage $\kappa$ while we lift the embedding, first through $\mathbb{P}_{\kappa}$ and then through $\mathbb{Q}$. 

Let $G_{\kappa} \subseteq \mathbb{P}_{\kappa}$ be $V$-generic, and $g \subseteq \mathbb{Q}$ be $V[G_{\kappa}]$-generic. First, we need an $M$-generic filter for $j(\mathbb{P}_{\kappa})$. Since cp$(j) = \kappa$, the forcings $\mathbb{P}_{\kappa}$ and $j(\mathbb{P}_{\kappa})$ agree up to stage $\kappa$. Since $\alpha < \kappa$, where $\alpha$ is the ordinal in the statement of the theorem, and the forcing $\mathbb{Q}$ adds a club $c_{\kappa} \subseteq \kappa$ such that $\delta \in c_{\kappa}$ implies $o(\delta)^V < \alpha$ which implies $o(\delta)^M < \alpha$, it follows that the $\kappa$th stage of $j(\mathbb{P}_{\kappa})$ is $\mathbb{Q}$. Thus, $j(\mathbb{P}_{\kappa})$ factors as $\mathbb{P}_{\kappa}*\dot{\mathbb{Q}}*\mathbb{P}_{\mbox{tail}}$ where $\mathbb{P}_{\mbox{tail}}$ is the forcing past stage $\kappa$. Since $G_{\kappa}*g$ is generic for $\mathbb{P}_{\kappa}*\dot{\mathbb{Q}}$, we just need an $M[G_{\kappa}][g]$-generic filter for $\mathbb{P}_{\mbox{tail}}$ which is in $V[G_{\kappa}][g]$ (since we want the final embedding to be there). Thus, diagonalize to get an $M[G_{\kappa}][g]$-generic filter for $\mathbb{P}_{\mbox{tail}}$ after checking we have met the criteria. Namely, since $|\mathbb{P}_{\kappa}| = |\dot{\mathbb{Q}}| = \kappa$, the forcing $\mathbb{P}_{\kappa}*\dot{\mathbb{Q}}$ has the $\kappa^+$-chain condition. Since $M^{\kappa} \subseteq M$, it follows by Hamkins [8 (Theorem 54)] that $M[G_{\kappa}][g]^{\kappa} \subseteq M[G_{\kappa}][g]$. Also, since $2^{\kappa} = \kappa^+$, the forcing $\mathbb{P}_{\kappa}$ has $\kappa^+$ many dense sets in $V$. Thus, $\mathbb{P}_{\mbox{tail}}$ has at most $|j(\kappa^+)|^V \le \kappa^{+^{\kappa}} = \kappa^+$ many dense sets in $M[G_{\kappa}][g]$. Also, since for any $\beta < \kappa$, the forcing $\mathbb{P}_{\kappa}$ has a dense subset which is $\le \beta$-closed (by Lemma 7); it follows that there is a dense subset of $\mathbb{P}_{\mbox{tail}}$ which is $\le \kappa$-closed. Thus, we may diagonalize to get $G^* \subset \mathbb{P}_{\mbox{tail}}$ for this dense subset, a generic filter in $V[G_{\kappa}][g]$, which is $M[G_{\kappa}][g]$-generic. Thus, $j(G_{\kappa}) = G_{\kappa}*g*G^*$ is $M$-generic for $j(\mathbb{P}_{\kappa})$ and $j''G_{\kappa} \subseteq G_{\kappa}*g*G^*$, so we may lift $j$ to $j: V[G_{\kappa}] \to M[j(G_{\kappa})]$. 

Next, we lift $j$ through $\mathbb{Q}$. The forcing $j(\mathbb{Q})$ adds a club $C \subseteq j(\kappa)$ such that $\delta \in C$ implies $o(\delta)^{M[j(G_{\kappa})} < \alpha$. Since $G_{\kappa}*g*G^* \in V[G_{\kappa}][g]$ and $M^{\kappa} \subseteq M$, it follows [Hamkins 8 (Theorem 53)] that $M[j(G_{\kappa})]^{\kappa} \subseteq M[j(G_{\kappa})].$ Also, by Lemma 7, there is a dense subset of $j(\mathbb{Q})$ which is $\le \kappa$-closed, and since $|\mathbb{Q}| = \kappa$, it follows that $j(\mathbb{Q})$ has at most $\kappa^+$ many dense subsets in $M[j(G_{\kappa})]$. Let $c = \cup g$, which is club in $\kappa$, and consider the set $\overline{c} = c \cup \{\kappa\} \in M[j(G_{\kappa})]$. We have $\delta \in c$ implies $o(\delta)^V < \alpha$. Since $M$ and $V$ agree on $P(P(\delta))$ for $\delta < \kappa$, it follows that $o(\delta)^V = o(\delta)^M$. By Lemma 9, we have $o(\delta)^{M[j(G_{\kappa})]} \le o(\delta)^M$. Thus $\delta \in c$ implies $o(\delta)^{M[j(G_{\kappa})]} < \alpha$. Since $o(\kappa)^M < \alpha$, it follows that $o(\kappa)^{M[j(G_{\kappa})]} < \alpha$, so that $\overline{c}$ is a condition of $j(\mathbb{Q})$, since $\overline{c}$ is a closed, bounded subset of $j(\kappa)$ such that $\delta \in \overline{c}$ implies $o(\delta)^{M[j(G_{\kappa})]} < \alpha$. Thus, diagonalize to get an $M[j(G)]$-generic filter, $g^* \subseteq j(\mathbb{Q})$ which contains condition $\overline{c}$. Then $j''g \subseteq j(g) = g^*$. Therefore, we may lift the embedding to $j:V[G_{\kappa}][g] \to M[j(G_{\kappa})][j(g)]$.  

From the previous lifting arguments will follow that $o(\kappa)^{V[G]} \ge \min\{ o(\kappa)^{V}, \alpha \}$. The lifted embedding $j: V \to M$ was chosen so that $o(\kappa)^{M} = \beta_0 < \min\{o(\kappa)^V, \alpha \}$. Since $\beta_0$ was arbitrary, it follows from Lemma 8 that $o(\kappa)^{V[G]} \ge \min\{o(\kappa)^V, \alpha\}$. All that remains to see is $o(\kappa)^{V[G]} \le \min\{o(\kappa)^V, \alpha \}$. By Lemma 9, which shows Mitchell rank does not go up for $V \subseteq V[G]$, we have $o(\kappa)^{V[G]} \le o(\kappa)^V$. Consider any embedding $j:V[G] \to M[j(G)]$ with critical point $\kappa$ in $V[G]$, and the new club $c \subseteq \kappa$, where $\forall \delta < \kappa$, $\delta \in c$ implies $o(\delta)^V < \alpha$, which is in any normal measure on $\kappa$. By the induction hypothesis, $\delta \in c$ implies $o(\delta)^{V[G]} < \alpha$. Applying $j$ to this statement, gives $\forall \delta < j(\kappa)$, $\delta \in j(c)$ implies $o(\delta)^{M[j(G)]} < \alpha$. Since $\kappa < j(\kappa)$ and $\kappa \in j(c)$, this gives $o(\kappa)^{M[j(G)]} < \alpha$. Thus by Lemma 8, since $j$ was arbitrary, $o(\kappa)^{V[G]} \le \alpha$. Thus, $o(\kappa)^{V[G]} = \min\{o(\kappa)^V, \alpha\}$.
\end{proof}

Let us now consider cases where $o(\kappa) \ge \kappa$. For this we use the concept of representing functions. Suppose $\kappa$ is measurable, $o(\kappa) = \alpha$, and $\alpha \in H_{\theta^+}$. The ordinal $\alpha$ is represented by $f:\kappa \to V_{\kappa}$ if whenever $j:V \to M$ is an elementary embedding with cp$(j) = \kappa$ and $M^{\kappa} \subseteq M$, then $j(f)(\kappa) = \alpha$ (this can be formalized as a first-order statement using extenders) by Hamkins [9]. Let $F:$ ORD $ \to $ ORD. The main theorem shows how to change the Mitchell rank for all measurable cardinals $\gamma$ for which $F \upharpoonright \gamma$ represents $F(\gamma)$ (through any $\gamma$-closed embedding with critical point $\gamma$).

\begin{theorem} For any $V \models $ ZFC + GCH and any $F:$ ORD $\to$ ORD, there is a forcing extension $V[G]$ where, if $\kappa$ is measurable and $F \upharpoonright \kappa$ represents $F(\kappa)$ in $V$, then $o(\kappa)^{V[G]} = \min\{o(\kappa)^V,F(\kappa)\}$.
\end{theorem}

\begin{proof} For example, we can do this for numerous functions satisfying this hypothesis, such as $F(\gamma) = \gamma, F(\gamma) = \gamma + 3, F(\gamma) = \gamma^2 + \omega^2 \cdot 2 + 5$, and any other function which can be defined like this. Suppose $V \models$ ZFC + GCH. Let $F:$ ORD $\to$ ORD. Let $\mathbb{P}$ be an Easton support ORD-length iteration, forcing at inaccessible stages $\gamma$, to add a club $C_{\gamma} \subset \gamma$ such that $\delta \in C_{\gamma}$ implies $o(\delta)^V < F(\delta)$. Let $\delta_0$ denote the first inaccessible cardinal. The forcing up to this stage, $\mathbb{P}_{\delta_0}$, adds nothing and the forcing $\mathbb{Q}$ at stage $\delta_0$ adds a club to $\delta_0$, so that $|\mathbb{P}_{\delta_0}*\dot{\mathbb{Q}}| \le \delta_0$. The forcing past stage $\delta_0$ is closed up to the next inaccessible cardinal. Let $G \subseteq \mathbb{P}$ be $V$-generic. Since $\mathbb{P}$ has a closure point at $\delta_0$, by Lemma 13 in Hamkins [11] it follows thath $V \subseteq V[G]$ satisfies the $\delta_0^+$ approximation and cover properties. Thus by Lemma 9 the Mitchell rank does not go up between $V$ and $V[G]$.

Let $\kappa$ be a measurable cardinal such that $F \upharpoonright \kappa$ represents $F(\kappa) = \alpha$ in $V$. Here we will need that representing functions are still representing functions in the extension. The fact we will use for this is from [Hamkins 11] which says that if $V \subseteq V[G]$ satisfies the hypothesis of the approximation and cover theorem, then every $\kappa$-closed embedding $j:V[G] \to M[j(G)]$ is a lift of an embedding $j \upharpoonright V: V \to M$ that is in $V$ (and $j \upharpoonright V$ is also a $\kappa$-closed embedding).
Thus, if $F \upharpoonright \kappa$ is a representing function for $F(\kappa)$ in the ground model, $(j \upharpoonright V)(F \upharpoonright \kappa)(\kappa) = F(\kappa) = \alpha$, and thus $F \upharpoonright \kappa$ is a representing function for $F(\kappa)$ with respect to such embeddings $j$ in $V[G]$.

The induction hypothesis is that for all $\kappa' < \kappa$ for which $F \upharpoonright \kappa'$ represents $F(\kappa')$ in $V$ has $o(\kappa')^{V[G]} = \min\{F(\kappa'), o(\kappa')^V\}$. By Lemma 8, fix an embedding $j:V \to M$ with cp$(j) = \kappa$ and $o(\kappa)^M = \beta_0 < \min\{o(\kappa)^V,F(\kappa)\}$. Note that $j(F)(\kappa) = j(F\upharpoonright \kappa)(\kappa) = F(\kappa)$. We shall lift $j$ through $\mathbb{P}$. The forcing above $\kappa$ does not affect $o(\kappa)$ since the forcing past stage $\kappa$ is closed up to the next inaccessible. Call $\mathbb{P}_{\kappa}*\dot{\mathbb{Q}}$ the forcing up to and including stage $\kappa$ while we lift the embedding, first through $\mathbb{P}_{\kappa}$ and then through $\mathbb{Q}$. Let $G_{\kappa} \subseteq \mathbb{P}_{\kappa}$ be $V$-generic, and $g \subseteq \mathbb{Q}$ be $V[G_{\kappa}]$-generic. First, we need an $M$-generic filter for $j(\mathbb{P}_{\kappa})$. Since cp$(j) = \kappa$, the forcings $\mathbb{P}_{\kappa}$ and $j(\mathbb{P}_{\kappa})$ agree up to stage $\kappa$. The forcing $\mathbb{Q}$ adds a club $C_{\kappa} \subseteq \kappa$ such that $\delta \in C_{\kappa}$ implies $o(\delta)^V < F(\delta)$. Since $j(F)(\kappa) = F(\kappa)$ and $o(\kappa)^M < F(\kappa)$, the $\kappa$th stage of $j(\mathbb{P}_{\kappa})$ is $\mathbb{Q}$. Thus, $j(\mathbb{P}_{\kappa})$ factors as $\mathbb{P}_{\kappa}*\dot{\mathbb{Q}}*\mathbb{P}_{\mbox{tail}}$ where $\mathbb{P}_{\mbox{tail}}$ is the forcing past stage $\kappa$. Since $G_{\kappa}*g$ is generic for $\mathbb{P}_{\kappa}*\dot{\mathbb{Q}}$, we just need an $M[G_{\kappa}][g]$-generic filter for $\mathbb{P}_{\mbox{tail}}$ which is in 
$V[G_{\kappa}][g]$ (since we want the final embedding to be there). Thus we shall diagonalize to get an $M[G_{\kappa}][g]$-generic filter for $\mathbb{P}_{\mbox{tail}}$ after checking we have met the criteria. Namely, since $|\mathbb{P}_{\kappa}| = |\mathbb{Q}| = \kappa$, the forcing $\mathbb{P}_{\kappa}*\dot{\mathbb{Q}}$ has the $\kappa^+$-chain condition. Since $M^{\kappa} \subseteq M$, it follows [Hamkins 8 (Theorem 54)] that $M[G_{\kappa}][g]^{\kappa} \subseteq M[G_{\kappa}][g]$. Also, since $2^{\kappa} = \kappa^+$, the forcing $\mathbb{P}_{\kappa}$ has $\kappa^+$ many dense sets in $V$. Thus, $\mathbb{P}_{\mbox{tail}}$ has at most $|j(\kappa^+)|^V \le \kappa^{+^{\kappa}} = \kappa^+$ many dense sets in $M[G_{\kappa}][g]$. Also, since $\beta < \kappa$, the forcing $\mathbb{P}_{\kappa}$ has a dense subset which is $\le \beta$-closed (by Lemma 7), it follows that there is a dense subset of $\mathbb{P}_{\mbox{tail}}$ which is $\le \kappa$-closed. Thus, we many diagonalize to get $G^* \subseteq \mathbb{P}_{\mbox{tail}}$ for this dense subset, a generic filter in $V[G_{\kappa}][g]$, which is $M[G_{\kappa}][g]$-generic. Thus, $j(G_{\kappa}) = G_{\kappa}*g*G'$ is $M$-generic for $j(\mathbb{P}_{\kappa})$ and $j''G_{\kappa} \subseteq G_{\kappa}*g*G^*$, so we may lift $j$ to $j: V[G_{\kappa}] \to M[j(G_{\kappa})]$.

Next, we lift $j$ through $\mathbb{Q}$. The forcing $j(\mathbb{Q})$ adds a club $C \subseteq j(\kappa)$ such that $\delta \in C$ implies $o(\delta)^{M[j(G_{\kappa})]} < j(F)(\delta)$. Since $j(G_{\kappa}) \in V[G_{\kappa}][g]$ and $M^{\kappa} \subseteq M$, it folows [Hamkins 8 (Theorem 53)] that $M[j(G_{\kappa})]^{\kappa} \subseteq M[j(G_{\kappa})]$. also, by Lemma 7, there is a dense subset of $j(\mathbb{Q})$ which is $\le \kappa$-closed, and since $|\mathbb{Q}| = \kappa$, it follows that $j(\mathbb{Q})$ has at most $\kappa^+$ many dense subsets in $M[j(G_{\kappa})]$. Let $c = \cup g$ and consider the set $\overline{c} = c \cup \{\kappa\}$. By the construction of the forcing, $\forall \delta < \kappa, \delta \in c$ implies $o(\delta)^V < F(\delta)$. Thus, $\forall \delta < \kappa, \delta \in c$ implies $o(\delta)^M < F(\delta)$. By Lemma 9, it follows that $\delta \in c$ implies $o(\delta)^{M[j(G)]} < F(\delta)$. Since $o(\kappa)^{M[j(G)]} \le o(\kappa)^M = \beta_0 < \min\{o(\kappa)^V, F(\kappa) = \alpha\}$ it follows that $o(\kappa)^{M[j(G)]} < \alpha = j(F)(\kappa)$. Thus $\overline{c}$ is a closed, bounded subset of $j(\kappa)$ such that if $\delta \in \overline{c}$ then $o(\delta)^{M[j(G)]} < j(F)(\delta)$. Thus $\overline{c}$ is a condition of $j(\mathbb{Q})$. Thus, diagonalize to get an $M[j(G)]$-generic filter, $g^* \subseteq j(\mathbb{Q})$ which contains the condition $\overline{c}$, so that $j''g \subseteq j(g) = g^*$. Therefore, we may lift the embedding to $j:V[G_{\kappa}][g] \to M[j(G_{\kappa})][j(g)]$.

From the previous lifting arguments will follow that $o(\kappa)^{V[G]} \ge \min\{o(\kappa)^V, F(\kappa)\}$. The lifted embedding $j: V \to M$ was chosen so that $o(\kappa)^M = \beta_0 < \min\{o(\kappa)^V, F(\kappa)\}$. Since $\beta_0$ was arbitrary, by Lemma 9 it follows that $o(\kappa)^{V[G]} \ge \min\{o(\kappa)^V, F(\kappa)\}$. All that is left is to see that $o(\kappa)^{V[G]} \le \min\{o(\kappa)^V, F(\kappa)\}$. By Lemma 9 which shows that Mitchell rank does not go up for $V \subseteq V[G]$, we have $o(\kappa)^{V[G]} \le o(\kappa)^V$. Consider any embedding $j:V[G] \to M[j(G)]$ in $V[G]$, and the new club $c \subseteq \kappa$ where $\forall \delta < \kappa, \delta \in c$ implies $o(\delta)^V < F(\delta)$. By the induction hypothesis, $\delta \in c$ implies $o(\delta)^{V[G]} < F(\delta)$. Apply $j$ to this stagement to get $\forall \delta < j(\kappa), \delta \in j(c)$ implies $o(\delta)^{M[j(G)]} < j(F)(\kappa) = F(\kappa)$. Note that since $F \upharpoonright \kappa$ represents $F(\kappa)$ in $V[G]$ and $M[j(G)]^{\kappa} \subseteq M[j(G)]$, it follows that $F \upharpoonright \kappa$ represents $F(\kappa)$ in $M[j(G)]$ as well. Since $j$ was arbitrary, $\kappa < j(\kappa)$, and $\kappa \in j(c)$ it now follows by Lemma 8 that $o(\kappa)^{V[G]} \le F(\kappa)$. Thus $o(\kappa)^{V[G]} \le \min\{o(\kappa)^V, F(\kappa)\}$.
\end{proof}

 \vspace{.5 in}

\begin{center}\textbf{SUPERCOMPACT AND STRONGLY COMPACT CARDINALS} \end{center}
\vspace{.2 in}

This section begins with the definitions of supercompact and strongly compact cardinals. The main theorem is about forcing to softly kill to any degree of supercompactness. An uncountable cardinal $\kappa$ is $\theta$-\textit{supercompact} if and only if there is an elementary embedding $j:V \to M$, with critical point $\kappa$ and $M^{\theta} \subseteq M$, where $\kappa < \theta < j(\kappa)$. Equivalently, $\kappa$ is $\theta$-supercompact if and only if there is a normal fine measure on $P_{\kappa}(\theta)$. A \textit{fine measure} is a $\kappa$-complete ultrafilter such that $\forall \sigma \in P_{\kappa} \theta$ the set $\{X \subseteq P_{\kappa} \theta \ | \ \sigma \in X\}$ is in the filter. A cardinal $\kappa$ is \textit{supercompact} if and only if it is $\theta$-supercompact for every $\theta > \kappa$. Notice that $\kappa$ measurable means that $\kappa$ is $\kappa$-supercompact. An uncountable cardinal $\kappa$ is $\theta$-\textit{strongly compact} if and only if there is an elementary embedding $j:V \to M$, with critical point $\kappa$, such that for all $t \subseteq M$, where $|t|^V \le \theta$, there is an $s \in M$ such that $t \subseteq s$ and $|s|^M < j(\kappa)$. Equivalently, a cardinal $\kappa$ is $\theta$-strongly compact for $\kappa < \theta$ if and only if there is a $\kappa$-complete fine measure on $P_{\kappa} \theta$. An uncountable cardinal $\kappa$ is \textit{strongly compact} if and only if it is $\theta$-strongly compact, for every $\theta > \kappa$. Equivalently, an uncountable regular cardinal $\kappa$ is strongly compact if and only if for any set $S$, every $\kappa$-complete filter on $S$ can be extended to a $\kappa$-complete ultrafilter on $S$ [Jech p. 365]. The following theorems show how to softly kill, to any level of supercompactness, to any degree of supercompactness, to any level of strong compactness, to measurable cardinals from supercompact cardinals, to measurable cardinals from strongly compact cardinals, and to strongly compact cardinals from supercompact cardinals.

\begin{theorem} If $\kappa$ is $<\theta$-supercompact, for $\theta \ge \kappa$ regular with $\beta^{<\kappa} < \theta$ for all $\beta < \theta$, then there is a forcing extension where $\kappa$ is $<\theta$-supercompact, but not $\theta$-supercompact, and indeed not even $\theta$-strongly compact.
\end{theorem}

\begin{proof} Suppose $\kappa$ is $<\theta$-supercompact for some regular $\theta > \kappa$, and let $\mathbb{P} = \mbox{Add}(\omega, 1)*\mbox{Add}(\theta,1)$. By Hamkins and Shelah [13], after forcing with $\mathbb{P}$, the cardinal $\kappa$ is not $\theta$-supercompact. However, adding a Cohen real is small relative to $\kappa$, so the first factor of $\mathbb{P}$ preserves that $\kappa$ is $<\theta$-supercompact. The second factor of $\mathbb{P}$ is $<\theta$-closed, so that a normal fine measure on $P_{\kappa} \beta$, where $\beta < \theta$, is still a normal fine measure on $P_{\kappa} \beta$, since no new subsets of $\beta$ were added. Let $G \subseteq \mathbb{P}$ be $V$-generic. Thus, in $V[G]$, the cardinal $\kappa$ is $<\theta$-supercompact, but not $\theta$-supercompact. 
\end{proof}

\begin{corollary} If $\kappa$ is $\theta$-supercompact where $\kappa < \theta$ and $\theta^{<\kappa} = \theta$, then there is a forcing extension where $\kappa$ is $\theta$-supercompact, but not $\theta^+$-supercompact.
\end{corollary}

\begin{theorem} If $\kappa$ is $<\theta$-strongly compact, for $\theta \ge \kappa$ regular with $\beta^{<\kappa} < \theta$ for all $\beta < \theta$, then there is a forcing extension where $\kappa$ is $<\theta$-strongly compact, but not $\theta$-strongly compact. 
\end{theorem}

\begin{proof} Suppose $\kappa$ is $<\theta$-strongly compact for regular $\theta \ge \kappa$, and let $\mathbb{P} = \mbox{Add}(\omega,1)*\mbox{Add}(\theta,1)$. By Hamkins and Shelah [13], this forcing destroys the $\theta$-strong compactness of $\kappa$. However, adding a Cohen real is small relative to $\kappa$, so the first factor preserves that $\kappa$ is $<\theta$-strongly compact. Also, since the second factor is $<\theta$-closed, andy $\kappa$-complete fine measure on $P_{\kappa}\beta$, where $\beta < \theta$, is still a $\kappa$-complete fine measure since there are no new subsets of $\beta$ to measure. Thus, $\mathbb{P}$ preserves that $\kappa$ is $<\theta$-strongly compact. Let $G \subseteq \mathbb{P}$ be $V$-generic. Then, in $V[G]$, the cardinal $\kappa$ is $<\theta$-strongly compact, but not $\theta$-strongly compact.
\end{proof}

\begin{corollary} If $\kappa$ is $\theta$-strongly compact, where $\kappa \le \theta$ and $\theta^{<\kappa} = \theta$, then there is a forcing extension where $\kappa$ is $\theta$-strongly compact, but not $\theta^+$-strongly compact.
\end{corollary}

\begin{corollary} If $\kappa$ is measurable, then there is a forcing extension where $\kappa$ is measurable but not strongly compact. 
\end{corollary}

\begin{proof} Let $\theta = \kappa$ in the proof of Theorem 13.
\end{proof}

The following theorem is due to Magidor [17]. It shows how to softly kill a supercompact cardinal so that it is still strongly compact in the forcing extension.

\begin{theorem} (Magidor) If $\kappa$ is strongly compact, then there is a forcing extension where $\kappa$ is strongly compact, but not supercompact.
\end{theorem}

A word on the proof of the theorem. Suppose $\kappa$ is a strongly compact cardinal. Then, there is a forcing extension, $V[G]$, where $\kappa$ is the least strongly compact cardinal, and the least measurable cardinal [17]. However, the least measurable cardinal can never be supercompact, since there are many measurable cardinals below any supercompact cardinal. Thus, in $V[G]$, the cardinal $\kappa$ is strongly compact, but not supercompact. 

As for measurable cardinals, an analogous rank to Mitchell rank can be assigned to suprcompactness embeddings. Suppose a cardinal $\kappa$ is $\theta$-supercompact for fixed $\theta$. If $\mu$ and $\nu$ are normal fine measures on $P_{\kappa}\theta$, define the Mitchell relation $\mu \vartriangleleft_{\theta-sc} \nu$ if and only if $\mu \in M_{\nu}$ where $j:V \to M_{\nu}$ is an ultrapower embedding by $\nu$. Since the relation $\vartriangleleft_{\theta-sc}$ is well-founded for given $\kappa$ and $\theta$, the Mitchell rank of $\mu$ is its rank with respect to $\vartriangleleft_{\theta-sc}$. Thus for fixed $\kappa$ and $\theta$, the notation $o_{\theta-sc}(\kappa) = \alpha$ means the height of $\vartriangleleft_{\theta-sc}$ on normal fine measures on $P_{\kappa}\theta$ is $\alpha$. By definition, $o_{\theta-sc}(\kappa)$ is the height of the well-founded Mitchell relation $\vartriangleleft_{\theta-sc}$ on $m(\kappa)$, the collection of normal fine measures on $P_{\kappa}\theta$. Thus, the definition of rank for $\mu \in m(\kappa)$ with respect to $\vartriangleleft_{\theta-sc}$ is $o_{\theta-sc}(\mu) = \sup\{o_{\theta-sc}(\nu) + 1 \ | \ \nu \vartriangleleft_{\theta-sc} \mu \}$. Thus $o_{\theta-sc}(\kappa) = \sup\{o_{\theta-sc}(\mu) + 1 \ | \ \mu \in m(\kappa) \}$. So that $o_{\theta-sc}(\kappa) = 0$ means that $\kappa$ is not $\theta$-supercompact, and $o_{\theta-sc}(\kappa) > 0$ means that $\kappa$ is $\theta$-supercompact. 

The main theorems of this section will rely on the following lemmas about Mitchell rank for supercompactness and are analogues of Lemmas 8, 9, and 10 for Mitchell rank.

\begin{lemma} If $\kappa$ is a $\theta$-supercompact cardinal, and $\mu$ is a normal fine measure on $P_{\kappa}\theta$, and $j_{\mu} : V \to M_{\mu}$ is the ultrapower embedding by $\mu$ with critical point $\kappa$, then $o_{\theta-sc}(\mu) = o_{\theta-sc}(\kappa)^{M_{\mu}}$. Thus, $\beta < o_{\theta-sc}(\kappa)$ if and only if there exists $j : V \to M$ elementary embedding with critical point $\kappa$, and $M^{\theta} \subseteq M$, with $o_{\theta-sc}(\kappa)^M = \beta$.
\end{lemma}

\begin{proof} Suppose $\kappa$ is $\theta$-supercompact, $\mu$ is a normal fine measure on $P_{\kappa}\theta$, and $j_{\mu}:V \to M_{\mu}$ is the ultrapower embedding by $\mu$ with critical point $\kappa$. The model $M_{\mu}$ computes $\vartriangleleft_{\theta-sc}$ correctly for $\nu \vartriangleleft_{\theta-sc} \mu$, which means $\nu \in M_{\mu}$, since $M^{\theta} \subseteq M$, and so $M_{\mu}$ has all the functions needed to check whether $\nu \vartriangleleft_{\theta-sc} \mu$. Then

\begin{equation}
\begin{split}
o_{\theta-sc}(\kappa)^{M_{\mu}} & = \sup\{o_{\theta-sc}(\nu)^{M_{\mu}} + 1 \ | \ \nu \in M_{\mu} \} \\
 & = \sup\{o_{\theta-sc}(\nu)^{M_{\mu}}+1 \ | \ \nu \vartriangleleft_{\theta-sc} \mu\}\\
 & = \sup\{o_{\theta-sc}(\nu) + 1 \ | \ \nu \vartriangleleft_{\theta-sc} \mu\} \\
 & = o_{\theta-sc}(\mu).
\end{split}
\end{equation}

Thus, the Mitchell rank for supercompactness of $\kappa$ in $M_{\mu}$ is the Mitchell rank of $\mu$. Let $\beta < o_{\theta-sc}(\kappa)$. Since $\vartriangleleft_{\theta-sc}$ is well-founded, there is $\mu$ a normal fine measure on $P_{\kappa}\theta$ such that $o_{\theta-sc}(\mu) = \beta$. Then if $j_{\mu}:V \to M_{\mu}$ is the ultrapower by $\mu$ with critical point $\kappa$, then $o_{\theta-sc}(\kappa)^{M_{\mu}} = o_{\theta-sc}(\mu) = \beta$. Conversely, suppose there is an elementary embedding $j:V \to M$ with cp$(j) = \kappa$ and $o_{\theta-sc}(\kappa)^M = \beta$. Let $\mu$ be the induced normal fine measure, using $j''\theta$ as a seed ($X \in \mu$ if and only if $j''\theta \in j(X)$) and let $M_{\mu}$ be the ultrapower by $\mu$. Then $M$ and $M_{\mu}$ have the same normal fine measures and the same Mitchell order and ranks for $\theta$-supercompactness, so $M_{\mu} \models o_{\theta-sc}(\kappa) = \beta$. Since we have established $o_{\theta-sc}(\kappa)^{M_{\mu}} = o_{\theta-sc}(\mu)$, the Mitchell rank for supercompactness of $\mu$ in $m(\kappa)$ is $\beta$, that is $o_{\theta-sc}(\mu) = \beta$. Thus since $o_{\theta-sc}(\kappa) = \sup\{o_{\theta-sc}(\mu) + 1 \ | \ \mu \in m(\kappa) \}$, it follows that $o_{\theta-sc}(\kappa)^V > \beta$. 
\end{proof}

The following lemma is an analogue of Lemma 9 for Mitchell rank for supercompactness. It is a generalization of Corollary 26 in [Hamkins 11] which states that if $V \subseteq \overline{V}$ satisfies the approximation and cover properties, then no supercompact cardinals are created from $V$ to $\overline{V}$. 

\begin{lemma} Suppose $V \subseteq \overline{V}$ satisfies the $\delta$ approximation and cover properties. If $\kappa, \theta > \delta$, then $o_{\theta-sc}(\kappa)^{\overline{V}} \le o_{\theta-sc}(\kappa)^V$.
\end{lemma}

\begin{proof} Suppose $V \subseteq \overline{V}$ satisfies the $\delta$ approximation and cover properties. Suppose $\kappa, \theta > \delta$. We want to show $o_{\theta-sc}(\kappa)^{\overline{V}} \le o_{\theta-sc}(\kappa)^V$ by induction on $\kappa$. Assume inductively that for all $\kappa' < \kappa$, we have $o_{\theta-sc}(\kappa')^{\overline{V}} \le o_{\theta-sc}(\kappa')^V$. If $o_{\theta-sc}(\kappa)^{\overline{V}} > o_{\theta-sc}(\kappa)^V$, then by Lemma 19 get a $\theta$-supercompactness embedding $j:\overline{V} \to \overline{M}$ with critical point $\kappa$ and $o_{\theta-sc}(\kappa)^{\overline{M}} = o_{\theta-sc}(\kappa)^V$. By [Hamkins 11], $j$ is a lift of a $\theta$-supercompactness embedding $j \upharpoonright V : V \to M$ for some $M$, and this embedding is a class in $V$. Again by Lemma 19, we have established that $o_{\theta-sc}(\kappa)^M < o_{\theta-sc}(\kappa)^V$. Thus, $o_{\theta-sc}(\kappa)^M < o_{\theta-sc}(\kappa)^{\overline{M}}$. However, this contradicts that $j$ applied to the induction hypothesis gives $o_{\theta-sc}(\kappa)^{\overline{M}} \le o_{\theta-sc}(\kappa)^M$ since $\kappa < j(\kappa)$. Therefore, $o_{\theta-sc}(\kappa)^{\overline{V}} \le o_{\theta-sc}(\kappa)^V$.
\end{proof}

\begin{lemma} If $V \subseteq V[G]$ is a forcing extension and every $\theta$-supercompactness embedding $j:V \to M$ in $V$ (with any critical point $\kappa$) liftls to an embedding $j: V[G] \to M[j(G)]$, then $o_{\theta-sc}(\kappa)^{V[G]} \ge o_{\theta-sc}(\kappa)^V$. In addition, if $V \subseteq V[G]$ also has $\delta$ approximation and $\delta$ cover properties, then $o_{\theta-sc}(\kappa)^{V[G]} = o_{\theta-sc}(\kappa)^V$.
\end{lemma}

\begin{proof} Suppose $V \subseteq V[G]$ is a forcing extension where every normal ultrapower embedding in $V$ lifts to a normal ultrapower embedding in $V[G]$. Suppose $\kappa$ is a $\theta$-supercompact cardinal, and for every $\theta$-supercompact cardinal $\kappa' < \kappa$, assume $o_{\theta-sc}(\kappa')^{V[G]} \ge o_{\theta-sc}(\kappa')^V$. Fix any $\beta < o_{\theta-sc}(\kappa)$. By Lemma 19, there exists an elementary embedding $j:V \to M$ with cp$(j) = \kappa$ and $M \models o_{\theta-sc}(\kappa) = \beta$. Then, by the assumption on $V \subseteq V[G]$, this $j$ lifts to $j:V[G] \to M[j(G)]$. Then, applying $j$ to the induction hypothesis gives $o_{\theta-sc}(\kappa)^{M[j(G)]} \ge o_{\theta-sc}(\kappa)^M$ since $\kappa < j(\kappa)$. Then, since $M \models o_{\theta-sc}(\kappa) = \beta$ it follows that $o_{\theta-sc}(\kappa)^{M[j(G)]} \ge \beta$. Thus $o_{\theta-sc}(\kappa)^{V[G]} \ge o_{\theta-sc}(\kappa)^V$.

If $V \subseteq V[G]$ also satisfies the $\delta$ approximation and $\delta$ cover properties, then by Lemma 20 the Mitchell rank does not go up between these models. Thus $o_{\theta-sc}(\kappa)^{V[G]} = o_{\theta-sc}(\kappa)^V$.
\end{proof}

The following theorem shows how to force a $\kappa^+$-supercompact cardinal $\kappa$ to have Mitchell rank for supercompactness at most 1.

\begin{theorem} If $V \models$ ZFC + GCH then there is a forcing extension where every cardinal $\kappa$ which is $\kappa^+$-supercompact has $o_{\kappa^+-sc}(\kappa)^{V[G]} = \min\{o_{\kappa^+-sc}(\kappa)^V, 1\}$.
\end{theorem}

\begin{proof} Suppose $V \models $ ZFC + GCH. Let $\mathbb{P}$ be an Easton support ORD-length iteration which forces at inaccessible stages $\gamma$ to add a club $c_{\gamma} \subseteq \gamma$ such that $\delta \in c_{\gamma}$ implies $o_{\delta^+-sc}(\delta)^V = 0$. Let $G \subseteq \mathbb{P}$ be $V$-generic. Let $\delta_0$ be the first inaccessible cardinal. The forcing before stage $\delta_0$ is trivial, and the forcing at stage $\delta_0$ adds a club to $\delta_0$. Thus the forcing $\mathbb{P}$ up to and including stage $\delta_0$ has cardinality $\delta_0$, and the forcing past stage $\delta_0$ has a dense set which is closed up to the next inaccessible cardinal. Thus $\mathbb{P}$ has a closure point at $\delta_0$ which implies $V \subseteq V[G]$ satisfies the $\delta_0$ approximation and cover properties by [Hamkins 11 (Lemma 13)]. Thus by [Hamkins 11 (Corollary 26)], the forcing $\mathbb{P}$ does not create $\theta$-supercompact cardinals for any $\theta$. Note that for any $\theta$-supercompact cardinal $\kappa$, both $\theta$ and $\kappa$ are above the closure point.

Suppose $\kappa$ is $\kappa^+$-supercompact. The induction hypothesis is that any $\kappa' < \kappa$ has $o_{\kappa'^+-sc}(\kappa')^{V[G]} = \min\{o_{\kappa'^+-sc}(\kappa')^V, 1\}$. Pick $j:V \to M$ a $\kappa^+$-supercompactness embedding with critical point $\kappa$ such that $o_{\kappa^+-sc}(\kappa)^M = 0 < \min\{o_{\kappa^+-sc}(\kappa)^V,1\}$. We shall lift $j$ through $\mathbb{P}$. The forcing above $\kappa$ does not affect the Mitchell rank for $\kappa^+$-supercompactness of $\kappa$ since the forcing past stage $\kappa$ has a dense subet which is closed up to the next inaccessible cardinal past $\kappa$. Call $\mathbb{P}_{\kappa}*\dot{\mathbb{Q}}$ the forcing up to and including stage $\kappa$. First, we lift $j$ through $\mathbb{P}_{\kappa}$. Let $G_{\kappa} \subseteq \mathbb{P}_{\kappa}$ be $V$-generic and $g \subseteq \mathbb{Q}$ be $V[G_{\kappa}]$-generic. First, we need an $M$-generic filter for $j(\mathbb{P}_{\kappa})$. Since cp$(j) = \kappa$, the forcings $\mathbb{P}_{\kappa}$ and $j(\mathbb{P}_{\kappa})$ agree up to stage $\kappa$. The forcing $\mathbb{Q}$ adds a club $c \subseteq \kappa$ such that $\forall \delta < \kappa, \delta \in c$ implies $o_{\delta^+-sc}(\delta)^V = 0$. Since $o_{\delta^+-sc}(\delta)^V = 0$ implies $o_{\delta^+-sc}(\delta)^M = 0$, it follows that the $\kappa$th stage of $j(\mathbb{P}_{\kappa})$ is $\mathbb{Q}$. Thus, $j(\mathbb{P}_{\kappa})$ factors as $\mathbb{P}_{\kappa}*\dot{\mathbb{Q}}*\mathbb{P}_{\mbox{tail}}$ where $\mathbb{P}_{\mbox{tail}}$ is the forcing past stage $\kappa$. Since $G_{\kappa}*g$ is $M$-generic for $\mathbb{P}_{\kappa}*\dot{\mathbb{Q}}$, we only need an $M[G_{\kappa}][g]$-generic filter for $\mathbb{P}_{\mbox{tail}}$ which is in $V[G_{\kappa}][g]$. We will diagonalize to obtain such a filter, after checking that we have met the criteria. Since $|\mathbb{P}_{\kappa}| = |\mathbb{Q}| = \kappa$, and $V \models $ GCH, it follows that the forcing $\mathbb{P}_{\kappa}*\dot{\mathbb{Q}}$ has the $\kappa^+$-chain condition. Since $M^{\kappa^+} \subseteq M$, it follows [Hamkins 8 (Theorem 54)] that $M[G_{\kappa}][g]^{\kappa^+} \subseteq M[G_{\kappa}][g]$. Since $|\mathbb{P}_{\kappa}| = \kappa$, and has the $\kappa^+$-chain condition, it has at most $\kappa^{<\kappa^+} = \kappa^+$ many maximal antichains. Since $|j(\kappa^+)| \le \kappa^{+^{{\kappa^+}^{<\kappa}}} = \kappa^{++}$, it follows that $\mathbb{P}_{\mbox{tail}}$ has at most $\kappa^{++}$ many maximal antichains in $M[G_{\kappa}][g]$. Since $\mathbb{P}_{\mbox{tail}}$ has a dense subset which is closed up to the next inaccessible cardinal past $\kappa$, it has a dense subset which is $\le \kappa^+$-closed. Thus, we can enumerate the maximal antichains of $\mathbb{P}_{\mbox{tail}}$, and build a descending sequence of conditions meeting them all through the dense set which has $\le \kappa^+$-closure, using this closure and the fact that $M[G_{\kappa}][g]$ is closed under $\kappa^+$-sequences to build through the limit stages. Then the upward closure of this sequence $G_{\mbox{tail}} \subseteq \mathbb{P}_{\mbox{tail}}$ is $M[G_{\kappa}][g]$-generic and is in $V[G_{\kappa}][g]$. Thus $j(G_{\kappa}) = G_{\kappa}*g*G_{\mbox{tail}}$ is $M$-generic, it is in $V[G_{\kappa}][g]$ and $j''G_{\kappa} \subseteq j(G_{\kappa})$. Thus, we have the partial lift $j:V[G_{\kappa}] \to M[j(G_{\kappa})]$.

Next we lift through $\mathbb{Q}$ by finding an $M[j(G_{\kappa})]$-generic filter for $j(\mathbb{Q})$. The forcing $j(\mathbb{Q})$ adds a club $C \subseteq j(\kappa)$ such that $\forall \delta < j(\kappa), \delta \in C$ implies $o_{\delta^+-sc}(\delta)^{M[j(G_{\kappa})]} = 0$. Since $j(G_{\kappa}) \in V[G_{\kappa}][g]$ and $M^{\kappa^+} \subseteq M$, it follows that $M[j(G_{\kappa})]^{\kappa^+} \subseteq M[j(G_{\kappa})]$ by [Hamkins 8 (Theorem 53)]. Since for every $\beta < \kappa$, there is a dense subset of $\mathbb{Q}$ which is $\le \beta$-closed, it follows that there is a dense subset of $j(\mathbb{Q})$ which is $\le \kappa^+$-closed. Also since $|\mathbb{Q}| = \kappa$ and $\mathbb{Q}$ has the $\kappa^+$-chain condition, the forcing $\mathbb{Q}$ has at most $\kappa^+$ many maximal antichains. Therefore $j(\mathbb{Q})$ has at most $|j(\kappa^+)| \le \kappa^{+^{\kappa^{+^{<\kappa}}}} = \kappa^{++}$ many maximal antichains in $M[j(G_{\kappa})]$. Thus we can diagonalize to get a generic filter, but we will do so with a master condition as follows. Let $c = \cup g$ be the new club of $\kappa$ and consider $\overline{c} = c \cup \kappa$ which is in $M[j(G_{\kappa})]$. We have $\delta \in c$ implies $o_{\delta^+-sc}(\delta)^V = 0$. Thus, $\delta \in c$ implies $o_{\delta^+-sc}(\delta)^M = 0$. By Corollary 26 [Hamkins 11], no new $\delta^+$-supercompact cardinals are created between $M$ and $M[j(G_{\kappa})]$, so that $\delta \in c$ implies $o_{\delta^+-sc}(\delta)^{M[j(G_{\kappa})]} = 0$. Since $o_{\kappa^+-sc}(\kappa)^M = 0$, it follows that $o_{\kappa^+-sc}(\kappa)^{M[j(G_{\kappa})]} = 0$. Thus $\overline{c}$ is a closed, bounded subset of $j(\kappa)$ such that every $\delta$ in $\overline{c}$ has $o_{\delta^+-sc}(\delta)^{M[j(G_{\kappa})]} = 0$. Thus $\overline{c} \in j(\mathbb{Q})$. Thus diagonalize to get an $M[j(G_{\kappa})]$-generic filter $g^* \subseteq j(\mathbb{Q})$ which contains the master condition $\overline{c}$. Then $j''g \subseteq j(g) = g^*$. Therefore we may lift embedding to $j:V[G_{\kappa}][g] \to M[j(G_{\kappa})][j(g)]$.

The success of the lifting arguments show that $o_{\kappa^+-sc}(\kappa)^{V[G]} \ge \min\{o_{\kappa^+-sc}(\kappa)^V, 1\}$. By Lemma 20, we have $o_{\kappa^+-sc}(\kappa)^{V[G]} \le o_{\kappa^+-sc}(\kappa)^V$. All that remains is to see that $o_{\kappa^+-sc}(\kappa)^{V[G]} \le 1$. Consider any $\kappa^+$-supercompactness embedding $j:V[G] \to M[j(G)]$ with critical point $\kappa$ in $V[G]$, and the new club $c \subseteq \kappa$, where $\forall \delta < \kappa, \delta \in c$ implies $o_{\kappa^+-sc}(\delta)^V = 0$. This club is in any normal fine measure on $P_{\kappa}\kappa^+$. By the induction hypothesis, $\forall \delta < \kappa, \delta \in c$ implies $o_{\delta^+-sc}(\delta)^{V[G]} = 0$. Applying $j$ to this statement gives $\forall \delta < j(\kappa), \delta \in j(c)$ implies $o_{\delta^+-sc}(\delta)^{M[j(G)]} = 0$. Thus by Lemma 19, since $j$ was arbitrary, $\kappa < j(\kappa)$, and $\kappa \in j(c)$, it follows that $o_{\kappa^+-sc}(\kappa)^{M[j(G)]} = 0$, and hence $o_{\kappa^+-sc}(\kappa)^{V[G]} \le 1$. Thus $o_{\kappa^+-sc}(\kappa)^{V[G]} \le \min\{o_{\kappa^+-sc}(\kappa)^V, 1\}$. Thus, for any $\kappa$ which is $\kappa^+$-supercompact, we have $o_{\kappa^+-sc}(\kappa)^{V[G]} = \min\{o_{\kappa^+-sc}(\kappa)^V,1\}$.
\end{proof}

For the following theorem for softly killing Mitchell rank for supercompactness, we will again need representing functions because we would like to now consider cases where $o_{\kappa^+-sc}(\kappa) \ge \kappa$. Then for the most general theorem, we would like to consider the cases where $\kappa$ is $\theta$-supercompact for $\theta > \kappa^+$. In [Hamkins 9] representing functions are considered for supercompactness embeddings as well. Let $\alpha \in H_{\theta^+}$ and $\kappa$ be a $\theta$-supercompact cardinal. The cardinal $\alpha$ is represented by $f: \kappa \to V_{\kappa}$ if whenever $j: V \to M$ an elementary embedding with cp$(j) = \kappa$ and $M^{\theta} \subseteq M$, then $j(f)(\kappa) = \alpha$ (this can be formalized as a first-order statement using extenders) [Hamkins 9]. Let $F:$ORD $\to$ ORD. The next theorem shows how to softly kill the Mitchell rank for supercompactness for all $\gamma^+$-supercompact cardinals $\gamma$ for which $F \upharpoonright \gamma$ represents $F(\gamma)$.

\begin{theorem} For any $V \models$ ZFC + GCH and any $F: ORD \to ORD$, there is a forcing extension $V[G]$ where, if $\kappa$ is $\kappa^+$-supercompact and $F \upharpoonright \kappa$ represents $F(\kappa)$ in $V$, then $o_{\kappa^+-sc}(\kappa)^{V[G]} = \min\{o_{\kappa^+-sc}(\kappa)^V, F(\kappa)\}$.
\end{theorem}

\begin{proof} Suppose $V \models$ ZFC + GCH. Let $F:$ ORD $\to$ ORD. Let $\mathbb{P}$ be an Easton support ORD-length iteration which forces at inaccessible stages $\gamma$ to add a club $c_{\gamma} \subseteq \gamma$ such that $\forall \delta < \gamma, \delta \in c_{\gamma}$ implies $o_{\delta^+-sc}(\delta)^V < F(\delta)$. Let $G \subseteq \mathbb{P}$ be $V$-generic. Let $\delta_0$ be the first inaccessible cardinal. The forcing before stage $\delta_0$ is trivial, and the forcing at stage $\delta_0$ adds a club to $\delta_0$. Thus the forcing $\mathbb{P}$ up to and including stage $\delta_0$ has cardinality $\delta_0$ and the forcing past stage $\delta_0$ has a dense set which is closed up to the next inaccessible cardinal. Thus $\mathbb{P}$ has a closure point at $\delta_0$, thus $V \subseteq V[G]$ satisfies the $\delta_0^+$ approximation and cover properties by Hamkins [11 (Lemma 13)]. Thus by Hamkins [11 (Corollary 26)], the forcing $\mathbb{P}$ does not create $\theta$-supercompact cardinals for any $\theta$. Note that for any $\theta$-supercompact cardinal $\kappa$, both $\kappa$ and $\theta$ are above this closure point.

Let $\kappa$ be a $\kappa^+$-supercompact cardinal such that $F \upharpoonright \kappa$ represents $F(\kappa) = \alpha$ in $V$. Here we will need that representing functions are still representing funtions in the extension. The fact we will use is from [11] which says that if $V \subseteq V[G]$ satisfies the hypothesis of approximation and cover theorem, then every $\theta$-closed embedding $j:V[G] \to M[j(G)]$ is a lift of an embedding $j\upharpoonright V: V \to M$ that is in $V$ (and $j \upharpoonright V$ is also a $\theta$-closed embedding). Thus, if $F \upharpoonright \kappa$ is a representing function for $F(\kappa)$ in the ground model, $(j \upharpoonright V)(F \upharpoonright \kappa)(\kappa) = F(\kappa)$, and thus $F \upharpoonright \kappa$ is a representing function for $F(\kappa)$ with respect to such embeddings $j$ in $V[G]$.

The induction hypothesis is that any $\kappa' < \kappa$, where $F \upharpoonright \kappa'$ represents $F(\kappa')$ in $V$, has $o_{\kappa'^+-sc}(\kappa')^{V[G]} = \min\{o_{\kappa'^+-sc}(\kappa')^V, F(\kappa')\}$. Pick $j:V \to M$ a $\kappa^+$-supercompactness embedding with critical point $\kappa$ such that $o_{\kappa^+-sc}(\kappa)^M = \beta_0 < \min\{o_{\kappa^+-sc}(\kappa)^V, F(\kappa)\}$. We shall lift $j$ through $\mathbb{P}$. The forcing above $\kappa$ does not affect the Mitchell rank for $\kappa^+$-supercompactness of $\kappa$ since the forcing past stage $\kappa$ has a dense subset which is closed up to the next inaccessible cardinal past $\kappa$. Call $\mathbb{P}_{\kappa}*\dot{\mathbb{Q}}$ the forcing up to and including stage $\kappa$.

First we shall lift $j$ through $\mathbb{P}_{\kappa}$. Let $G_{\kappa} \subseteq \mathbb{P}_{\kappa}$ be $V$-generic and $g \subseteq \mathbb{Q}$ be $V[G_{\kappa}]$-generic. First, we an $M$-generic filter for $j(\mathbb{P}_{\kappa})$. Since cp$(j) = \kappa$, the forcings $\mathbb{P}_{\kappa}$ and $j(\mathbb{P}_{\kappa})$ agree up to stage $\kappa$. The forcing $\mathbb{Q}$ adds a club $c \subseteq \kappa$ such that $\forall \delta < \kappa, \delta \in c$ implies $o_{\delta^+-sc}(\delta)^V < F(\delta)$. Since $o_{\delta^+-sc}(\delta)^V < F(\delta)$ implies $o_{\delta^+-sc}(\delta)^M < F(\delta) = j(F)(\delta)$, it follows that the $\kappa$th stage of $j(\mathbb{P}_{\kappa})$ is $\mathbb{Q}$. Thus, $j(\mathbb{P}_{\kappa})$ factors as $\mathbb{P}_{\kappa}*\dot{\mathbb{Q}}*\mathbb{P}_{\mbox{tail}}$ where $\mathbb{P}_{\mbox{tail}}$ is the forcing past stage $\kappa$. Since $G_{\kappa}*g$ is $M$-generic for $\mathbb{P}_{\kappa}*\dot{\mathbb{Q}}$, we only need an $M[G_{\kappa}][g]$-generic filter for $\mathbb{P}_{\mbox{tail}}$ which is in $V[G_{\kappa}][g]$. We will diagonalize to obtain such a filter, after checkiing that we have met the criteria. Since $|\mathbb{P}_{\kappa}| = |\mathbb{Q}| = \kappa$, and $V \models$ GCH, it follows that the forcing $\mathbb{P}_{\kappa}*\dot{\mathbb{Q}}$ has the $\kappa^+$-chain condition. Since $M^{\kappa^+} \subseteq M$, it follows [Hamkins 8 (Theorem 54)] that $M[G_{\kappa}][g]^{\kappa^+} \subseteq M[G_{\kappa}][g]$. Since $|\mathbb{P}_{\kappa}| = \kappa$, and has the $\kappa^+$-chain condition, it has at most $\kappa^{<\kappa^+} = \kappa^+$ many maximal antichains. Since $|j(\kappa^+)| \le \kappa^{+^{\kappa^{+^{<\kappa}}}} = \kappa^{++}$it follows that $\mathbb{P}_{\mbox{tail}}$ has at most $\kappa^{++}$ many maximal antichains in $M[G_{\kappa}][g]$. Since $\mathbb{P}_{\mbox{tail}}$ has a dense subset which is closed up to the next inaccessible past $\kappa$, it has a dense subset which is $\le \kappa^+$-closed. Thus we can enumerate the maximal antichains of $\mathbb{P}_{\mbox{tail}}$, and build a descending sequence of conditions meeting them all through the dense set which has $\le \kappa^+$-closure, using this closure and the fact that $M[G_{\kappa}][g]$ is closed under $\kappa^+$-sequences to build through the limit stages. Then the upward closure of this sequence $G_{\mbox{tail}} \subseteq \mathbb{P}_{\mbox{tail}}$ is $M[G_{\kappa}][g]$-generic and is in $V[G_{\kappa}][g]$. Thus $j(G_{\kappa}) = G_{\kappa}*g*G_{\mbox{tail}}$ is $M$-generic, it is in $V[G_{\kappa}][g]$ and $j''G_{\kappa} \subseteq j(G_{\kappa})$. Thus we have the partial lift $j:V[G_{\kappa}] \to M[j(G_{\kappa})]$.

Next we lift through $\mathbb{Q}$ by finding an $M[j(G_{\kappa})]$-generic filter for $j(\mathbb{Q})$. The forcing $j(\mathbb{Q})$ adds a club $C \subseteq j(\kappa)$ such that $\forall \delta < j(\kappa), \delta \in C$ implies $o_{\delta^+-sc}(\delta)^{M[j(G_{\kappa})]} < j(F)(\delta)$. Since $j(G_{\kappa}) \in V[G_{\kappa}][g]$ and $M^{\kappa^+} \subseteq M$, it follows that $M[j(G_{\kappa})]^{\kappa^+} \subseteq M[j(G_{\kappa})]$ by [8 (Theorem 53)]. Since for every $\beta < \kappa$, there is a dense subset of $\mathbb{Q}$ which is $\le \beta$-closed, it follows that there is a dense subset of $j(\mathbb{Q})$ which is $\le \kappa^+$-closed. Also, since $|\mathbb{Q}| = \kappa$ and $\mathbb{Q}$ has the $\kappa^+$-chain condition, the forcing $\mathbb{Q}$ has at most $\kappa^+$ many maximal antichains. Therefore $j(\mathbb{Q})$ has at most $|j(\kappa^+)| \le \kappa^{+^{\kappa^{+^{< \kappa}}}} = \kappa^{++}$ many maximal antichains in $M[j(G_{\kappa})]$. Thus we can diagonalize to get a generic filter, but we will do so with a master condition as follows. Let $c = \cup g$ be the new club of $\kappa$ and consider $\overline{c} = c \cup \{\kappa\}$ which is in $M[j(G_{\kappa})]$. We have $\delta \in c$ implies $o_{\delta^+-sc}(\delta)^M < F(\delta)$. By Lemma 20, the Mitchell rank for supercompactness does no go up between $M$ and $M[j(G_{\kappa})]$, so that $\delta \in c$ implies $o_{\delta^+-sc}(\delta)^{M[j(G_{\kappa})]} < F(\delta) = j(F)(\delta)$. Since $o_{\kappa^+-sc}(\kappa)^{M[j(G)]}\le o_{\kappa^+-sc}(\kappa)^M = \beta_0 < \min\{o_{\kappa^+-sc}(\kappa)^V, F(\kappa)\}$ it follows that $o_{\kappa^+}(\kappa)^{M[j(G)]} < F(\kappa) = j(F)(\kappa)$. Thus $\overline{c}$  is a closed, bounded subset of $j(\kappa)$ such that every $\delta$ in $\overline{c}$ has $o_{\delta^+-sc}(\delta)^{M[j(G_{\kappa})]} < j(F)(\kappa)$. Thus $\overline{c} \in j(\mathbb{Q})$. Thus diagonalize to get an $M[j(G_{\kappa})]$-generic filter $g^* \subseteq j(\mathbb{Q})$ which contains the master condition $\overline{c}$. Then $j''g \subseteq j(g) = g^*$. Therefore we may lift the embedding to $j:V[G_{\kappa}][g] \to M[j(G_{\kappa})][j(g)]$.

The success of the lifting arguments for $o_{\kappa^+-sc}(\kappa)^M = \beta_0 < \min\{o_{\kappa^+-sc}(\kappa)^V, F(\kappa)\}$, where $\beta_0$ was arbitrary, show that $o_{\kappa^+-sc}(\kappa)^{V[G]} \ge \min\{o_{\kappa^+-sc}(\kappa)^V, F(\kappa)\}$. By Lemma 20, we have $o_{\kappa^+-sc}(\kappa)^{V[G]} \le o_{\kappa^+-sc}(\kappa)^V$. All that remains is to see that $o_{\kappa^+-sc}(\kappa)^{V[G]} \le F(\kappa)$. Consider any $\kappa^+$-supercompact embedding $j:V[G] \to M[j(G)]$ with critical point $\kappa$ in $V[G]$, and the new club $c \subseteq \kappa$, where $\forall \delta < \kappa, \delta \in c$ implies $o_{\delta^+-sc}(\delta)^V < F(\delta)$. This club in any normal fine measure on $P_{\kappa}\kappa^+$. By the induction hypothesis, $\forall \delta, \kappa, \delta \in c$ implies $o_{\delta^+-sc}(\delta)^{V[G]} < F(\delta)$. Applying $j$ to this statement gives $\forall \delta < j(\kappa), \delta \in j(c)$ implies $o_{\delta^+-sc}(\delta)^{M[j(G)]} < j(F)(\delta)$. Since $F \upharpoonright \kappa$ represents $F(\kappa)$, we have $j(F)(\kappa)= F(\kappa)$. Since $\kappa < j(\kappa)$ and $\kappa \in j(c)$ it follows that $o_{\kappa^+-sc}(\kappa)^{M[j(G)]} < F(\kappa)$. Thus, by Lemma 19, since $j$ was arbitrary, $o_{\kappa^+-sc}(\kappa)^{V[G]} \le F(\kappa)$. Thus $o_{\kappa^+-sc}(\kappa)^{V[G]} \le \min\{o_{\kappa^+-sc}(\kappa)^V, F(\kappa)\}$. Thus $o_{\kappa^+-sc}(\kappa)^{V[G]} \le \min\{o_{\kappa^+-sc}(\kappa)^V, F(\kappa)\}$. Thus for any $\kappa$ which is $\kappa^+$-supercompact for which $F \upharpoonright \kappa$ represents $F(\kappa)$ in $V$ has $o_{\kappa^+}(\kappa)^{V[G]} = \min\{o_{\kappa^+-sc}(\kappa)^V, F(\kappa)\}$.
\end{proof}

The most general theorem considers the case of softly killing Mitchell rank for $\theta$-supercompactness for $\theta$-supercompact cardinals $\kappa$ when $\theta > \kappa^+$.

\begin{theorem} For any $V \models ZFC + GCH$, any $\Theta: ORD \to ORD$ and any $F: ORD \to ORD$, there is a forcing extension $V[G]$ where, if $\kappa$ is $\Theta(\kappa)$-supercompact, $\Theta \upharpoonright \kappa$ represents $\Theta(\kappa)$, $F \upharpoonright \kappa$ represents $F(\kappa)$ in $V$, and $\Theta''\kappa \subseteq \kappa$, then $o_{\Theta(\kappa)-sc}(\kappa)^{V[G]} = \min\{o_{\Theta(\kappa)-sc}(\kappa)^V, F(\kappa)\}$.
\end{theorem}

\begin{proof} Suppose $V \models$ ZFC + GCH. Let $\Theta:$ ORD $\to$ ORD and $F:$ ORD $\to$ ORD. Let $\mathbb{P}$ be an Easton support ORD-length iteration which forces at stage $\gamma$ to add a club $c_{\gamma} \subseteq \gamma$ such that $\delta \in c_{\gamma}$ implies $o_{\Theta(\delta)-sc}(\delta)^V < F(\delta)$, whenever $\gamma$ is a a closure point of $\Theta$, meaning $\Theta''\gamma \subseteq \gamma$, otherwise force trivially. Let $G \subseteq \mathbb{P}$ be $V$-generic. Let $\delta_0$ be the first inaccessible cardinal. The forcing before stage $\delta_0$ is trivial, and the forcing at stage $\delta_0$ adds a club to $\delta_0$. Thus, the forcing $\mathbb{P}$ up to and including stage $\delta_0$ has cardinality $\delta_0$ and the forcing past stage $\delta_0$ has a dense set which is closed up to the next inaccessible cardinal. Thus, $\mathbb{P}$ has a closure point at $\delta_0$, thus $V \subseteq V[G]$ satisfies the $\delta_0^+$ approximation and cover properties by [Hamkins 11 (Lemma 13)]. Thus, by Lemma 20, the forcing $\mathbb{P}$ does not increase Mitchell rank for $\theta$-supercompactness for cardinals above $\delta_0$ where $\theta$ is also above $\delta_0$. 

Let $\kappa$ be a $\Theta(\kappa)$-supercompact cardinal such that $\Theta''\kappa \subseteq \kappa$, and $\Theta \upharpoonright \kappa$ represents $\theta(\kappa)$, and $F \upharpoonright \kappa$ represents $F(\kappa)$ in $V$. Here we will need that representing functions are still representing functions in the extension. The fact we will use for this is from [Hamkins 11] which says that if $V \subseteq V[G]$ satisfies the hypothesis of the approximation and cover theorem, then every $\Theta(\kappa)$-closed embedding $j:V[G] \to M[j(G)]$ is a lift of an embedding $j \upharpoonright V : V \to M$ that is in $V$ (and $j \upharpoonright V$ is also a $\Theta(\kappa)$-closed embedding). Thus, if $F \upharpoonright \kappa$ is a representing function for $F(\kappa)$ in the ground model, $(j \upharpoonright V)(F \upharpoonright \kappa)(\kappa) = F(\kappa)$, and thus $F \upharpoonright \kappa$ is a representing function for $F(\kappa)$ with respect to such embeddings $j$
 in $V[G]$. 

 The induction hypothesis is that any $\kappa' < \kappa$, where $\Theta''\kappa' \subseteq \kappa'$, and $\Theta \upharpoonright \kappa'$ represents $\Theta(\kappa)$, and $F \upharpoonright \kappa'$ represents $F(\kappa')$ in $V$, has $o_{\Theta(\kappa)-sc}(\kappa')^{V[G]} = \min\{o_{\Theta(\kappa)-sc}(\kappa')^V, F(\kappa')\}$. Pick $j:V \to M$ a $\Theta(\kappa)$-supercompactness embedding with critical point $\kappa$ such that $o_{\Theta(\kappa)-sc}(\kappa)^M = \beta_0 < \min\{o_{\Theta(\kappa)-sc}(\kappa)^V, F(\kappa)\}$. We will lift $j$ through $\mathbb{P}$. The forcing above $\kappa$ does not affect the Mitchell rank for $\Theta(\kappa)$-supercompactness of $\kappa$ since the forcing past stage $\kappa$ has a dense subset which is closed up to the next closure point of $\Theta$ past $\kappa$. Call $\mathbb{P}_{\kappa}*\dot{\mathbb{Q}}$ the forcing up to and including stage $\kappa$.

 First we will lift $j$ through $\mathbb{P}_{\kappa}$. Let $G_{\kappa} \subset \mathbb{P}_{\kappa}$ be $V$-generic and $g \subseteq \mathbb{Q}$ be $V[G_{\kappa}]$-generic. First, we need an $M$-generic filter for $j(\mathbb{P}_{\kappa})$. Since cp$(j) = \kappa$, the forcings $\mathbb{P}_{\kappa}$ and $j(\mathbb{P}_{\kappa})$ agree up to stage $\kappa$. The forcing $\mathbb{Q}$ adds a club $c \subseteq \kappa$ such that $\forall \delta < \kappa, \delta \in c$ implies $o_{\Theta(\kappa)-sc}(\delta)^V < F(\delta)$. Since $o_{\Theta(\kappa)-sc}(\kappa)^V < F(\kappa)$ implies $o_{\Theta(\kappa)-sc}(\kappa)^M < F(\kappa) = j(F)(\kappa)$, and $j(\Theta)(\kappa) = \Theta(\kappa)$, it follows that the $\kappa$th stage of $j(\mathbb{P}_{\kappa})$ is $\mathbb{Q}$. Thus, $j(\mathbb{P}_{\kappa})$ factors as $\mathbb{P}_{\kappa}*\dot{\mathbb{Q}}*\mathbb{P}_{\mbox{tail}}$ where $\mathbb{P}_{\mbox{tail}}$ is the forcing past stage $\kappa$. Since $G_{\kappa}*g$ is $M$-generic for $\mathbb{P}_{\kappa}*\dot{\mathbb{Q}}$, we only need an $M[G_{\kappa}][g]$-generic filter for $\mathbb{P}_{\mbox{tail}}$ which is in $V[G_{\kappa}][g]$. We will diagonalize to obtain such a filter, after checking that we have met the criteria. Since $|\mathbb{P}_{\kappa}|=|\mathbb{Q}| = \kappa$, and $V \models $ GCH, it follows that the forcing $\mathbb{P}_{\kappa}*\dot{\mathbb{Q}}$ has the $\kappa^+$-chain condition. Since $M^{\Theta(\kappa)} \subseteq M$, it follows [Hamkins 8 (Theorem 54)] that $M[G_{\kappa}][g]^{\Theta(\kappa)} \subseteq M[G_{\kappa}][g]$. Since $|\mathbb{P}_{\kappa}| = \kappa$, and has the $\kappa^+$-chain condition, it has at most $\kappa^{<\kappa^+} = \kappa^+$ many maximal antichains. Since $|j(\kappa^+)| \le \kappa^{+^{\Theta(\kappa)^{< \kappa}}} = \Theta(\kappa)^+$, it follows that $\mathbb{P}_{\mbox{tail}}$ has at most $\Theta(\kappa)^+$ many maximal antichains in $M[G_{\kappa}][g]$. Since $\mathbb{P}_{\mbox{tail}}$ has a dense subset which is closed up to the next closure point of $\Theta$ past $\kappa$, it has a dense subset which is $\le \Theta(\kappa)$-closed. Thus, we can enumerate the maximal antichains of $\mathbb{P}_{\mbox{tail}}$, and build a descending sequence of conditions meeting them all through the dense set which has $\le \Theta(\kappa)$-closure, using this closure and the fact that $M[G_{\kappa}][g]$ is closed under $\Theta(\kappa)$-sequences to build through the limit stages. Then the upward closure of this sequence $G_{\mbox{tail}} \subseteq \mathbb{P}_{\mbox{tail}}$ is $M[G_{\kappa}][g]$-generic and is in $V[G_{\kappa}][g]$. Thus, $j(G_{\kappa}) = G_{\kappa}*g*G_{\mbox{tail}}$ is $M$-generic, it is in $V[G_{\kappa}][g]$ and $j''G_{\kappa} \subseteq j(G_{\kappa})$. Thus, we have the partial lift $j: V[G_{\kappa}] \to M[j(G_{\kappa})]$.

 Next we lift though $\mathbb{Q}$ by finding an $M[j(G_{\kappa})]$-generic filter for $j(\mathbb{Q})$. The forcing $j(\mathbb{Q})$ adds a club $C \subseteq j(\kappa)$ such that $\forall \delta < j(\kappa), \delta \in C$ implies $o_{\Theta(\kappa)-sc}(\delta)^{M[j(G_{\kappa})]} <j(F)(\delta)$. Since $j(G_{\kappa}) \in V[G_{\kappa}][g]$ and $M^{\Theta(\kappa)} \subseteq M$, it follows that $M[j(G_{\kappa})]^{\Theta(\kappa)} \subseteq M[j(G_{\kappa})]$ by [Hamkins 8 (Theorem 53)]. Since for every $\beta < \kappa$, there is a dense subset of $\mathbb{Q}$ which is $\le \beta$-closed, it follows that there is a dense subset of $j(\mathbb{Q})$ which is $\le \Theta(\kappa)$-closed. Also, since $|\mathbb{Q}| = \kappa$ and $\mathbb{Q}$ has the $\kappa^+$-chain condition, the forcing $\mathbb{Q}$ has at most $\kappa^+$ many maximal antichains. Therefore $j(\mathbb{Q})$ has at most $|j(\kappa^+)| \le \kappa^{+^{\Theta(\kappa)^{<\kappa}}} = \Theta(\kappa)^+$ many maximal antichains in $M[j(G_{\kappa})]$. Thus we can diagonalize to get a generic filter with a master condition as follows. Let $c = \cup g$ be the new club of $\kappa$ and consider $\overline{c} = c \cup \{ \kappa \}$ which is in $M[j(G_{\kappa})]$. We have $\delta \in c$ implies $o_{\Theta(\kappa)-sc}(\delta)^V < F(\delta)$. Thus, $\delta \in c$ implies $o_{\Theta(\kappa)-sc}(\delta)^M < F(\delta)$. By Lemma 20, it follows that $\delta \in c$ implies $o_{\Theta(\kappa)-sc}(\delta)^{M[j(G_{\kappa})]} \le o_{\Theta(\kappa)-sc}(\delta)^M < F(\delta) = j(F)(\delta)$. Since $o_{\Theta(\kappa)-sc}(\kappa)^{M[j(G)]} \le o_{\Theta(\kappa)-sc}(\kappa)^M = \beta_0 < \min\{o_{\Theta(\kappa)-sc}(\kappa)^V, F(\kappa)\}$, and $j(\Theta)(\kappa) = \Theta(\kappa)$, it follows that $o_{j(\Theta)(\kappa)-sc}(\kappa)^{M[j(G)]} < F(\kappa) = j(F)(\kappa)$. Thus $\overline{c}$ is a closed, bounded subset of $j(\kappa)$ such that every $\delta$ in $\overline{c}$ has $o_{j(\Theta)(\kappa)-sc}(\delta)^{M[j(G_{\kappa})]} < j(F)(\kappa)$. Thus $\overline{c} \in j(\mathbb{Q})$. Thus diagonalize to get an $M[j(G_{\kappa})]$-generic filter $g^* \subseteq j(\mathbb{Q})$ which contains the master condition $\overline{c}$. Then $j''g \subseteq j(g) = g^*$. Therefore we may lift the embedding to $j:V[G_{\kappa}][g] \to M[j(G_{\kappa})][j(g)]$.

 The success of the lifting arguments for $o_{\Theta(\kappa)-sc}(\kappa)^M = \beta_0 < \min\{o_{\Theta(\kappa)-sc}(\kappa)^V, F(\kappa)\}$, where $\beta_0$ was arbitrary, shows that $o_{\Theta(\kappa)-sc}(\kappa)^V[G] \ge \min\{o_{\Theta(\kappa)-sc}(\kappa)^V, F(\kappa)\}$. By Lemma 20, we have $o_{\Theta(\kappa)-sc}(\kappa)^{V[G]} \le o_{\Theta(\kappa)-sc}(\kappa)^V$. All that remains to see is that $o_{\Theta(\kappa)-sc}(\kappa)^{V[G]} \le F(\kappa)$. Consider any $\Theta(\kappa)$-supercompactness embedding $j:V[G] \to M[j(G)]$ with critical point $\kappa$ in $V[G]$, and the new club $c \subseteq \kappa$, where $\forall \delta < \kappa, \delta \in c$ implies $o_{\Theta(\delta)-sc}(\delta)^V < F(\delta)$. This club is in any normal fine measure on $P_{\kappa}(\Theta(\kappa))$. By the induction hypothesis, $\forall \delta < \kappa, \delta \in c$ implies $o_{\Theta(\delta)-sc}(\delta)^{V[G]} < F(\delta)$. Applying $j$ to this statement gives $\forall \delta < j(\kappa), \delta \in j(c)$ implies $o_{j(\Theta)(\kappa)-sc}(\delta)^{M[j(G)]} < j(F)(\delta)$. Since $F \upharpoonright \kappa$ represents $F(\kappa)$, we have $j(F)(\kappa) = F(\kappa)$, and since $\Theta \upharpoonright \kappa$ represents $\Theta(\kappa)$, we have $j(\Theta)(\kappa) = \Theta(\kappa)$. Since $\kappa < j(\kappa)$ and $\kappa \in j(c)$, it follows that $o_{\Theta(\kappa)-sc}(\kappa)^{M[j(G)]} < F(\kappa)$. Thus by Lemma 19, since $j$ was arbitrary, $o_{\Theta(\kappa)-sc}(\kappa)^{V[G]} \le F(\kappa)$. Thus $o_{\Theta(\kappa)-sc}(\kappa)^{V[G]} \le \min\{o_{\Theta(\kappa)-sc}(\kappa)^V, F(\kappa)\}.$ Thus for any $\kappa$ which is $\Theta(\kappa)$-supercompact for which $\Theta''\kappa \subseteq \kappa$, and $F \upharpoonright \kappa$ represents $F(\kappa)$, and $\Theta \upharpoonright \kappa$ represents $\Theta(\kappa)$ in $V$ has $o_{\Theta(\kappa)-sc}(\kappa)^{V[G]} = \min\{o_{\Theta(\kappa)-sc}(\kappa)^V, F(\kappa)\}$.
 \end{proof}

 The last theorem is about large cardinal properties close to the top. A cardinal $\kappa$ is \textit{huge} with target $\lambda$ if there is an elementary embedding $j:V \to M$, with critical point $\kappa$, such that $j(\kappa) = \lambda$ and $M^{\lambda} \subseteq M$. A cardinal $\kappa$ is \textit{superhuge} if it is huge with target $\lambda$ for unboundedly many cardinals. 

 \begin{theorem} If $\kappa$ is huge with target $\lambda$, then there is a forcing extension where $\kappa$ is still huge with target $\lambda$, but $\kappa$ is not superhuge.
 \end{theorem}

 \begin{proof} Suppose $\kappa$ is huge with target $\lambda$. Force with Add$(\omega,1)*$Add$(\lambda^+,1)$. Then, by Hamkins and Shelah [13], in the extension, $\kappa$ is not superhuge, since it is not even $\lambda^+$-supercompact, hence not $\lambda^+$-huge. But, $\kappa$ is still huge with target $\lambda$ since the forcing is $\le \lambda$-closed.
\end{proof}
\vspace{.5 in}
\begin{center}\textbf{OPEN QUESTIONS:}\end{center} 
\vspace{.2 in}
\noindent 1. Can we define $\varepsilon_0$-inaccessible cardinals? 

\noindent 2. If $\kappa$ is a $\Sigma_{n + 1}$-reflecting cardinal, is there $V[G]$ where $\kappa$ is $\Sigma_n$-reflecting, but not $\Sigma_{n+1}$ reflecting?

\noindent 3. Are there killing-them-softly results for $\eta$-extentible cardinals? How about their targets?

\noindent 4. Does every large cardinal submit to degrees? For those that do, can we softly kill them? What do the forcings that do so say about the degrees?

\end{document}